\documentclass[11pt]{article}

\usepackage{amsmath,amsthm,amssymb,amsfonts,charter,csquotes,enumitem,geometry,hyperref,mathrsfs,mathtools,shuffle}
\usepackage{todonotes}
\usepackage{tikz}
\usetikzlibrary{matrix}
\usetikzlibrary{cd}

\usepackage{fullpage}
\usepackage{lscape}

\usepackage{caption} 
\theoremstyle{definition}
\newtheorem{definition}{Definition}[section]
\newtheorem{theorem}[definition]{Theorem}

\newtheorem{corollary}[definition]{Corollary}
\newtheorem{proposition}[definition]{Proposition}

\newtheorem{remark}[definition]{Remark}
\newtheorem{conjecture}[definition]{Conjecture}

\newtheorem*{definition*}{Definition}
\newtheorem*{theorem*}{Theorem}
\newtheorem*{lemma*}{Lemma}
\newtheorem*{corollary*}{Corollary}
\newtheorem*{proposition*}{Proposition}
\newtheorem*{example*}{Example}
\newtheorem*{remark*}{Remark}
\newtheorem*{conjecture*}{Conjecture}

\newtheorem*{acknowledgments*}{Acknowledgments}

\begin{document}

\title{Cohomological Arithmetic Statistics for Principally Polarized Abelian Varieties over Finite Fields}
\author{Aleksander Shmakov}
\date{}

\maketitle


\begin{abstract}
There is a natural probability measure on the set of isomorphism classes of principally polarized Abelian varieties of dimension \(g\) over \(\mathbb{F}_q\), weighted by the number of automorphisms. The distributions of the number of \(\mathbb{F}_q\)-rational points are related to the cohomology of fiber powers of the universal family of principally polarized Abelian varieties. To that end we compute the cohomology \(H^i(\mathcal{X}^{\times n}_g,\mathbb{Q}_\ell)\) for \(g=1\) using results of Eichler-Shimura and for \(g=2\) using results of Lee-Weintraub and Petersen, and we compute the compactly supported Euler characteristics \(e_\mathrm{c}(\mathcal{X}^{\times n}_g,\mathbb{Q}_\ell)\) for \(g=3\) using results of Hain and conjectures of Bergstr\"om-Faber-van der Geer. In each of these cases we identify the range in which the point counts \(\#\mathcal{X}^{\times n}_g(\mathbb{F}_q)\) are polynomial in \(q\). Using results of Borel and Grushevsky-Hulek-Tommasi on cohomological stability, we adapt arguments of Achter-Erman-Kedlaya-Wood-Zureick-Brown to pose a conjecture about the asymptotics of the point counts \(\#\mathcal{X}^{\times n}_g(\mathbb{F}_q)\) in the limit \(g\rightarrow\infty\).
\end{abstract}

\section*{Introduction}

	Let \([\mathcal{A}_g(\mathbb{F}_q)]\) be the set of isomorphism classes of principally polarized Abelian varieties of dimension \(g\) over \(\mathbb{F}_q\). The cardinality \(\#[\mathcal{A}_g(\mathbb{F}_q)]\) is finite; of course, for each \([A,\lambda]\in[\mathcal{A}_g(\mathbb{F}_q)]\) the cardinality \(\#A(\mathbb{F}_q)\) is finite, and is constant in its isogeny class. One would like to understand how the point counts of principally polarized Abelian varieties over \(\mathbb{F}_q\) distribute. 
	
	Experience informs us that such point counting problems are better behaved when weighted by the number of automorphisms. To that end let \(\mathcal{A}_g(\mathbb{F}_q)\) be the groupoid of principally polarized Abelian varieties of dimension \(g\) over \(\mathbb{F}_q\). Consider the groupoid cardinality 
	\begin{align*}
		\#\mathcal{A}_g(\mathbb{F}_q)=\sum_{[A,\lambda]\in[\mathcal{A}_g(\mathbb{F}_q)]}\frac{1}{\#\mathrm{Aut}_{\mathbb{F}_q}(A,\lambda)}
	\end{align*}
	For example, one has (classically for \(g=1\), by Lee-Weintraub \cite[Corollary 5.2.3]{LeeWeintraub} for \(g=2\) and by Hain \cite[Theorem 1]{Hain} for \(g=3\)): 
	\begin{align*}
		\#\mathcal{A}_1(\mathbb{F}_q) &= q\\
		\#\mathcal{A}_2(\mathbb{F}_q) &= q^3+q^2\\
		\#\mathcal{A}_3(\mathbb{F}_q) &= q^6+q^5+q^4+q^3+1
	\end{align*}
	Consider the natural probability measure \(\mu_{\mathcal{A}_g(\mathbb{F}_q)}\) on \([\mathcal{A}_g(\mathbb{F}_q)]\) such that \([A,\lambda]\in[\mathcal{A}_g(\mathbb{F}_q)]\) has mass weighted by the number of automorphisms:  
	\begin{align*}
		\mu_{\mathcal{A}_g(\mathbb{F}_q)}([A,\lambda])=\frac{1}{\#\mathcal{A}_g(\mathbb{F}_q)\#\mathrm{Aut}_{\mathbb{F}_q}(A,\lambda)}
	\end{align*}
	On the discrete probability space \(([\mathcal{A}_g(\mathbb{F}_q)],2^{[\mathcal{A}_g(\mathbb{F}_q)]},\mu_{\mathcal{A}_g(\mathbb{F}_q)})\) consider the random variable \(\#A_g(\mathbb{F}_q):[\mathcal{A}_g(\mathbb{F}_q)]\rightarrow\mathbb{Z}\) assigning to \([A,\lambda]\in[\mathcal{A}_g(\mathbb{F}_q)]\) the point count \(\#A(\mathbb{F}_q)\). Our goal is to understand, among other things, the expected values \(\mathbb{E}(\#A_g(\mathbb{F}_q))\), and more generally the higher moments \(\mathbb{E}(\#A_g(\mathbb{F}_q)^n)\) with respect to the natural probability measure \(\mu_{\mathcal{A}_g(\mathbb{F}_q)}\). 
	
	For example, one has the expected values (classically for \(g=1\), by Lee \cite[Corollary 1.4]{Lee} for \(g=2\), and by \ref{Theorem3} for \(g=3\)): 
	\begin{align*}
		\mathbb{E}(\#A_1(\mathbb{F}_q)) &= q+1\\
		\mathbb{E}(\#A_2(\mathbb{F}_q)) &= q^2+q+1-\frac{1}{q^3+q^2}\\ 
		\mathbb{E}(\#A_3(\mathbb{F}_q)) &= q^3+q^2+q+1-\frac{q^2+q}{q^6+q^5+q^4+q^3+1}
	\end{align*}
	and one has the expected values (classically for \(g=1\), by Lee \cite[Corollary 1.5]{Lee} for \(g=2\), and by \ref{Theorem3} for \(g=3\)): 
	\begin{align*}
		\mathbb{E}(\#A_1(\mathbb{F}_q)^2) &= q^2+3q+1-\frac{1}{q}\\ 
		\mathbb{E}(\#A_2(\mathbb{F}_q)^2) &= q^4+3q^3+6q^2+3q-\frac{5q^2+5q+3}{q^3+q^2}\\
		\mathbb{E}(\#A_3(\mathbb{F}_q)^2) &= q^6+3q^5+6q^4+10q^3+6q^2+2q-2-\frac{8q^5+14q^4+12q^3+7q^2-2q-7}{q^6+q^5+q^4+q^3+1}
	\end{align*}
	Many more expected values are computed and displayed in \ref{MGF1}, \ref{MGF2}, and \ref{MGF3} later in the paper. 
	
	The above expected values are obtained by applying the Grothendieck-Lefschetz trace formula to the \(\ell\)-adic cohomology of the universal family of principally polarized Abelian varieties in order to produce the required point counts over finite fields. Let \(\mathcal{A}_g\) be the moduli of principally polarized Abelian varieties of dimension \(g\) and let \(\pi:\mathcal{X}_g\rightarrow\mathcal{A}_g\) be the universal family of Abelian varieties over \(\mathcal{A}_g\). Consider the \(n\)-fold fiber product: 
	\begin{align*}
		\pi^n:\mathcal{X}^{\times n}_g:=\underbrace{\mathcal{X}_g\times_{\mathcal{A}_g}\hdots\times_{\mathcal{A}_g}\mathcal{X}_g}_n\rightarrow\mathcal{A}_g
	\end{align*}
	Then the expected value \(\mathbb{E}(\#A_g(\mathbb{F}_q)^n)\) is related to the groupoid cardinality \(\#\mathcal{X}^{\times n}_g(\mathbb{F}_q)\):
	\begin{align*}
		\mathbb{E}(\#A_g(\mathbb{F}_q)^n)=\frac{\#\mathcal{X}^{\times n}_g(\mathbb{F}_q)}{\#\mathcal{A}_g(\mathbb{F}_q)}
	\end{align*}
	In order to compute the groupoid cardinalities \(\#\mathcal{X}^{\times n}_g(\mathbb{F}_q)\) it is enough to compute the compactly supported Euler characteristic \(e_\mathrm{c}(\mathcal{X}^{\times n}_g,\mathbb{Q}_\ell):=\sum_{i\geq 0}(-1)^iH^i_\mathrm{c}(\mathcal{X}^{\times n}_g,\mathbb{Q}_\ell)\) as an element of the Grothendieck group of \(\ell\)-adic Galois representations, in which case by applying the Grothendieck-Lefschetz trace formula we have:
	\begin{align*}
		\#\mathcal{X}^{\times n}_g(\mathbb{F}_q)=\mathrm{tr}(\mathrm{Frob}_q|e_\mathrm{c}(\mathcal{X}^{\times n}_g,\mathbb{Q}_\ell)):=\sum_{i\geq 0}\mathrm{tr}(\mathrm{Frob}_q|H^i_\mathrm{c}(\mathcal{X}^{\times n}_g,\mathbb{Q}_\ell))
	\end{align*}
	Note that since \(\mathcal{X}^{\times n}_g\) is the complement of a normal crossings divisor of a smooth proper Deligne-Mumford stack over \(\mathbb{Z}\) (see \cite[Chapter VI, Theorem 1.1]{FaltingsChai}), the \(\ell\)-adic \'etale cohomology \(H^i(\mathcal{X}^{\times n}_{g,\overline{\mathbb{Q}}},\mathbb{Q}_\ell)\) is unramified for all primes \(p\neq\ell\) (so that the action of \(\mathrm{Frob}_p\) is well-defined) and is isomorphic to the \(\ell\)-adic \'etale cohomology \(H^i(\mathcal{X}^{\times n}_{g,\overline{\mathbb{F}}_p},\mathbb{Q}_\ell)\) as a representation of \(\mathrm{Gal}(\overline{\mathbb{F}}_p/\mathbb{F}_p)\), with the action of \(\mathrm{Gal}(\overline{\mathbb{Q}}_p/\mathbb{Q}_p)\subseteq\mathrm{Gal}(\overline{\mathbb{Q}}/\mathbb{Q})\) factoring through the surjection \(\mathrm{Gal}(\overline{\mathbb{Q}}_p/\mathbb{Q}_p)\rightarrow\mathrm{Gal}(\overline{\mathbb{F}}_p/\mathbb{F}_p)\). Consequently we will use the cohomology over \(\overline{\mathbb{Q}}\) and the cohomology over \(\overline{\mathbb{F}}_p\) somewhat interchangeably, dropping either of these fields from the subscript whenever stating results which are true for both of these situations, as we have done above. 
	
	The computation requires three results: the first result \ref{Leray}, due to Deligne, involves the degeneration of the Leray spectal sequence computing \(H^*(\mathcal{X}^{\times n}_g,\mathbb{Q}_\ell)\) in terms of the cohomology of the \(\ell\)-adic local systems \(\mathbb{R}^j\pi^n_*\mathbb{Q}_\ell\) on \(\mathcal{A}_g\), the second result \ref{Kunneth} expresses the local systems \(\mathbb{R}^j\pi^n_*\mathbb{Q}_\ell\) in terms of the local systems \(\mathbb{V}_\lambda\) on \(\mathcal{A}_g\) corresponding to the irreducible representation of \(\mathrm{Sp}_{2g}\) of highest weight \(\lambda\), and the third result (\ref{EichlerShimura1} for \(g=1\) due to Eichler-Shimura, \ref{EichlerShimura2} for \(g=2\) due to Lee-Weintraub and Petersen, and \ref{EichlerShimura3} for \(g=3\) due to Hain and Bergstr\"om-Faber-van der Geer) computes the cohomology of \(\ell\)-adic cohomology of the local systems \(\mathbb{V}_\lambda\) on \(\mathcal{A}_g\). These results about the cohomology of local systems relies on the work of many people and results of the Langlands program as input. 
	
	Indeed, the expected values displayed so far might give the impression that the compactly supported Euler characteristics \(e_\mathrm{c}(\mathcal{X}^{\times n}_g,\mathbb{Q}_\ell)\) are Tate type, so that the point counts \(\#\mathcal{X}^{\times n}_g(\mathbb{F}_q)\) are polynomial in \(q\). This is not true in general: the compactly supported Euler characteristics \(e_\mathrm{c}(\mathcal{X}^{\times n}_g,\mathbb{Q}_\ell)\) in general involve \(\ell\)-adic Galois representations attached to vector-valued Siegel modular forms for \(\mathrm{Sp}_{2g}(\mathbb{Z})\), so that the point counts \(\#\mathcal{X}^{\times n}_g(\mathbb{F}_q)\) in general involve traces of Hecke operators on spaces of vector-valued Siegel modular forms. The relation between traces of Frobenius and traces of Hecke operators is ultimately obtained by the Langlands-Kottwitz method by comparing the Grothendieck-Lefschetz trace formula to the stabilization of the Arthur-Selberg trace formula \cite{Kottwitz}; while this strategy is overly sophisticated in the case \(g=1\), it is the strategy used in the work of Petersen \cite{Petersen} in the case \(g=2\) and by unpublished work of Ta\"ibi \cite{TaibiAg} in the case \(g\geq 3\). 

\paragraph{Summary of Results}
	
	For \(g=1,2\) we know enough about the cohomology of local systems on \(\mathcal{A}_g\) to compute \(H^i(\mathcal{X}^{\times n}_g,\mathbb{Q}_\ell)\) as an \(\ell\)-adic Galois representation (up to semisimplification). In the case \(g=1\) a classical result of Eichler-Shimura (see for example \cite[Theorem 2.3]{BFvdG3}) implies the following result:
	
	\begin{theorem*}
	\ref{Theorem1} The cohomology \(H^i(\mathcal{X}^{\times n}_1,\mathbb{Q}_\ell)\) is Tate type for all \(i\) and all \(1\leq n\leq 9\). The cohomology \(H^i(\mathcal{X}^{\times 10}_1,\mathbb{Q}_\ell)\) is Tate type for all \(i\neq 11\), whereas for \(i=11\) we have
	\begin{align*}
		H^{11}(\mathcal{X}^{\times 10}_1,\mathbb{Q}_\ell)=\mathbb{S}_{\Gamma(1)}[12]+\mathbb{L}^{11}+99\mathbb{L}^{10}+1925\mathbb{L}^9+12375\mathbb{L}^8+29700\mathbb{L}^7
	\end{align*}
	where \(\mathbb{S}_{\Gamma(1)}[12]\) is the \(2\)-dimensional \(\ell\)-adic Galois representation attached to the weight \(12\) cusp form \(\Delta\in S_{12}(\Gamma(1))\). In particular the  compactly supported Euler characteristic \(e_\mathrm{c}(\mathcal{X}^{\times n}_1,\mathbb{Q}_\ell)\) is not Tate type if \(n\geq 10\). 
	\end{theorem*}
	
	In the case \(g=2\) results of Lee-Weintraub \cite[Corollary 5.2.3]{LeeWeintraub} and Petersen \cite[Theorem 2.1]{Petersen} imply following result: 
	
	\begin{theorem*}
	\ref{Theorem2} The cohomology \(H^i(\mathcal{X}^{\times n}_2,\mathbb{Q}_\ell)\) is Tate type for all \(i\) and all \(1\leq n\leq 6\). The cohomology \(H^i(\mathcal{X}^{\times 7}_2,\mathbb{Q}_\ell)\) is Tate type for all \(i\neq 17\), whereas for \(i=17\) we have
	\begin{align*}
		H^{17}(\mathcal{X}^{\times 7}_2,\mathbb{Q}_\ell)=\mathbb{S}_{\Gamma(1)}[18]+\mathbb{L}^{17}+1176\mathbb{L}^{15}+63700\mathbb{L}^{13}+6860\mathbb{L}^{12}+321048\mathbb{L}^{11}+294440\mathbb{L}^{10}+\mathbb{L}^9
	\end{align*}
	where \(\mathbb{S}_{\Gamma(1)}[18]\) is the \(2\)-dimensional \(\ell\)-adic Galois representation attached to the weight \(18\) cusp form \(f_{18}=\Delta E_6\in S_{18}(\Gamma(1))\). In particular the  compactly supported Euler characteristic \(e_\mathrm{c}(\mathcal{X}^{\times n}_2,\mathbb{Q}_\ell)\) is not Tate type if \(n\geq 7\). 
	\end{theorem*}
	
	The cohomology groups \(H^i(\mathcal{X}^{\times n}_1,\mathbb{Q}_\ell)\) for \(1\leq n\leq 10\) and \(H^i(\mathcal{X}^{\times n}_2,\mathbb{Q}_\ell)\) for \(1\leq n\leq 7\) are displayed in the tables 1 and 2 at the end of the paper. The Euler characteristics \(e_\mathrm{c}(\mathcal{X}^{\times n}_1,\mathbb{Q}_\ell)\) for \(1\leq n\leq 10\) and \(e_\mathrm{c}(\mathcal{X}^{\times n}_2,\mathbb{Q}_\ell)\) for \(1\leq n\leq 7\) are displayed along with these theorems later in the paper. 
	
	In the case \(g=3\) there are precise conjectures of Bergstr\"om-Faber-van der Geer \cite[Conjecture 7.1]{BFvdG3} about the compactly supported Euler characteristics of local systems on \(\mathcal{A}_3\) as an element of the Grothendieck group of \(\ell\)-adic Galois representations. These conjectures are now known at least for small highest weight \(\lambda\) using dimension formulas for spaces of vector-valued Siegel modular forms for \(\mathrm{Sp}_6(\mathbb{Z})\) obtained by Ta\"ibi \cite{Taibi}. These conjectures, along with a result of Hain \cite[Theorem 1]{Hain} implies the following result: 
	
	\begin{theorem*}
	\ref{Theorem3} Assume conjectures \ref{EichlerShimura3} and \ref{MiddleGalois}. Then the Euler characteristic \(e_\mathrm{c}(\mathcal{X}^{\times n}_3,\mathbb{Q}_\ell)\) is Tate type for all \(1\leq n\leq 5\). The compactly supported Euler characteristic \(e_\mathrm{c}(\mathcal{X}^{\times 6}_3,\mathbb{Q}_\ell)\) is given by:
	\begin{align*}
		e_\mathrm{c}(\mathcal{X}^{\times 6}_3,\mathbb{Q}_\ell) &= (\mathbb{L}^6+21\mathbb{L}^5+120\mathbb{L}^4+280\mathbb{L}^3+309\mathbb{L}^2+161\mathbb{L}+32)\mathbb{S}_{\Gamma(1)}[0,10]\\
			&+ \mathbb{L}^{24}+22\mathbb{L}^{23}+253\mathbb{L}^{22}+2024\mathbb{L}^{21}+11362\mathbb{L}^{20}+46613\mathbb{L}^{19}\\
			&+ 146665\mathbb{L}^{18}+364262\mathbb{L}^{17}+720246\mathbb{L}^{16}+1084698\mathbb{L}^{15}+1036149\mathbb{L}^{14}+38201\mathbb{L}^{13}\\
			&- 1876517\mathbb{L}^{12}-3672164\mathbb{L}^{11}-4024657\mathbb{L}^{10}-2554079\mathbb{L}^9+101830\mathbb{L}^8+2028655\mathbb{L}^7\\
			&+ 2921857\mathbb{L}^6+2536864\mathbb{L}^5+1553198\mathbb{L}^4+687157\mathbb{L}^3+215631\mathbb{L}^2+45035\mathbb{L}+4930 
	\end{align*}
	where \(\mathbb{S}_{\Gamma(1)}[0,10]=\mathbb{S}_{\Gamma(1)}[18]+\mathbb{L}^9+\mathbb{L}^8\) is the \(4\)-dimensional \(\ell\)-adic Galois representation attached to the Saito-Kurokawa lift \(\chi_{10}\in S_{0,10}(\Gamma(1))\) of the weight \(18\) cusp form \(f_{18}=\Delta E_6\in S_{18}(\Gamma(1))\). In particular the compactly supported Euler characteristic \(e_\mathrm{c}(\mathcal{X}^{\times n}_3,\mathbb{Q}_\ell)\) is not Tate type if \(n\geq 6\).
	\end{theorem*}

	The Euler characteristics \(e_\mathrm{c}(\mathcal{X}^{\times n}_3,\mathbb{Q}_\ell)\) for \(1\leq n\leq 6\) are displayed along with these theorems later in the paper. In view of \cite[Theorem 1.9]{CanningLarsonPayne2}, using the classification results of Chevevier-Ta\"ibi \cite{ChenevierTaibi}, these computations are unconditional for \(1\leq n\leq 3\) on the basis of point counts. 
	
	We have continued these computations until reaching the first modular contributions: in the case \(g=1\) the contribution is through the discriminant cusp form \(\Delta\in S_{12}(\Gamma(1))\) which contributes the irreducible \(2\)-dimensional \(\ell\)-adic Galois representation \(\mathbb{S}_{\Gamma(1)}[12]\), and in the case \(g=2\) and \(g=3\) the contributions are through the Saito-Kurokawa lift \(\chi_{10}\in S_{0,10}(\Gamma(1))\) which contributes the irreducible \(2\)-dimensonal \(\ell\)-adic Galois representation \(\mathbb{S}_{\Gamma(1)}[18]\). One can continue further, where for \(g=2\), in the case \(n=11\) we have contributions from the vector-valued Siegel modular forms \(\chi_{6,8}\in S_{6,8}(\Gamma(1))\) and \(\chi_{4,10}\in S_{4,10}(\Gamma(1))\) of general type (see \cite[Section 25]{vdG} for the relevant dimensions), which contribute the irreducible \(4\)-dimensional \(\ell\)-adic Galois representations \(\mathbb{S}_{\Gamma(1)}[6,8]\) and \(\mathbb{S}_{\Gamma(1)}[4,10]\) (see \cite[Theorem I, Theorem II]{Weissauer}). For \(g=3\), in the case \(n=9\) we have a contribution from an \(8\)-dimensional \(\ell\)-adic Galois representation \(\mathbb{S}_{\Gamma(1)}[3,3,7]\) which decomposes into a \(1\)-dimensional \(\ell\)-adic Galois representation of Tate type and an irreducible \(7\)-dimensional \(\ell\)-adic Galois representation (see \cite[Example 9.1]{BFvdG3}), which is explained by a functorial lift from the exceptional group \(\mathrm{G}_2\) predicted by \cite{GrossSavin}. This is to say that if one continues a bit further, one encounters more complicated \(\ell\)-adic Galois representations in cohomology governing these arithmetic statistics. We end up using each of these contributions to deduce that \(e_\mathrm{c}(\mathcal{X}^{\times n}_g,\mathbb{Q}_\ell)\) is not Tate type above a certain range. 
	
	Nevertheless, it is reasonable to conjecture that these modular contributions to arithmetic statistics are negligible. As explained in \cite{AEKWB}, random matrix heuristics plausibly apply in the limit \(g\rightarrow\infty\) to the Frobenius eigenvalues of \(e_\mathrm{c}(\mathcal{X}^{\times n}_g,\mathbb{Q}_\ell)\) not explained by the existence of algebraic cycles, and bounding the traces of these matrices with high probability leads one to the heuristic that only certain Tate classes contribute to the Grothendieck-Lefschetz trace formula asymptotically. 

	Following this strategy, we pose the following conjecture about the distributions of the point counts \(\#A_g(\mathbb{F}_q)\) in the limit \(g\rightarrow\infty\): 
	
	\begin{conjecture*}
	\ref{Conjecture1} (compare to \cite[Conjecture 1]{AEKWB}) Let \(\lambda=1+\frac{1}{q}+\frac{1}{q(q-1)}=\frac{1}{1-q^{-1}}\). For all \(n\geq 1\) we have
	\begin{align*}
		\lim_{g\rightarrow\infty}q^{-ng}\mathbb{E}(\#A_g(\mathbb{F}_q)^n)=\lambda^{\frac{n(n+1)}{2}}
	\end{align*}
	\end{conjecture*}
	
	We pose a second conjecture \ref{Conjecture2} about the negligible contribution of certain to these point counts (compare to \cite[Heuristic 2]{AEKWB}) and show that this implies the first conjecture. The computations done in the cases \(g=1\), \(g=2\), and \(g=3\) are consistent with this conjecture in their respective stability ranges. 

\paragraph{Relation to Other Work}

	Much work has been done regarding the cohomology of local systems on \(\mathcal{M}_{g,n}\) and its compactification (see \cite{Pandharipande} for a survey, and for example \cite{Bergstrom}, \cite{BergstromFaber}, \cite{BergstromFaberPayne}, \cite{BergstromTommasi}, \cite{CanningLarsonPayne1}, \cite{CanningLarsonPayne2}, \cite{ChanGalatiusPayne}, \cite{ChanFaberGalatiusPayne}, \cite{Getzler}, \cite{MadsenWeiss}), and likewise for \(\mathcal{A}_g\) and its compactifications (see \cite{HulekTommasi} for a survey, and for example \cite{BFvdG3} \cite{ChenLooijenga}, \cite{GrushevskyHulek}, \cite{GrushevskyHulekTommasi}, \cite{Hain}, \cite{HulekTommasi1}, \cite{HulekTommasi2}, \cite{LeeWeintraub}, \cite{Petersen}). 

	The method we have used to investigate arithmetic statistics for varieties over finite fields is hardly new: it is explained very clearly by Lee \cite{Lee} in the case \(g=2\), where the computations of \(H^i(\mathcal{X}_2,\mathbb{Q}_\ell)\) and \(H^i(\mathcal{X}^{\times 2}_2,\mathbb{Q}_\ell)\) appear. The computations in the case \(g=3\) are new, but use the same method. The theme of identifying in which range modular contributions appear in the cohomology of fiber powers of the universal Abelian variety represents a departure from this previous work. 
	
	The work of Achter-Erman-Kedlaya-Wood-Zureick-Brown \cite{AEKWB} concerns the point counts \(\#\mathcal{M}_{g,n}(\mathbb{F}_q)\) in the limit \(g\rightarrow\infty\), and uses results of Madsen-Weiss \cite{MadsenWeiss} on cohomological stability for \(\mathcal{M}_{g,n}\) to show that the distributions of the point counts \(\#C_g(\mathbb{F}_q)\) are asymptotically Poisson with mean \(q\lambda=q+1+\frac{1}{q-1}=\frac{1}{q-1}\), assuming a conjecture on the negligible contribution of non-tautological classes to point counts. We have used the same method to study the point counts \(\#\mathcal{X}^{\times n}_g(\mathbb{F}_q)\) in the limit \(g\rightarrow\infty\), using results of Borel \cite{BorelI}, \cite{BorelII} and Grushevsky-Hulek-Tommasi \cite{GrushevskyHulekTommasi} on cohomological stability for \(\mathcal{X}^{\times n}_g\) to study the asymptotics of the distributions of the point counts \(\#A_g(\mathbb{F}_q)\), assuming an analogous conjecture on the negligible contribution of unstable classes to point counts. 
	
	The work of Achter-Altu\v{g}-Garcia-Gordon \cite{AAGG} takes a rather different approach to the study arithmetic statistics for principally polarized Abelian varieties over \(\mathbb{F}_q\), starting from a theorem of Kottwitz relating masses of isogeny classes to volumes of tori and twisted orbital integrals, and then relating these to a product of local factors \(\nu_v([A,\lambda],\mathbb{F}_q)\) over all places \(v\) of \(\mathbb{Q}\). By contrast, almost every result we have used about the Galois action on the \(\ell\)-adic cohomology of local systems on \(\mathcal{A}_g\) relies on the Langlands-Kottwitz method relating traces of Frobenius to traces of Hecke operators, starting from the same theorem of Kottwitz and ultimately relating this to the stabilization of the Arthur-Selberg trace formula. It may be interesting to relate these two approaches, for instance by reexamining the computations in this paper in terms of explicit computations of twisted orbital integrals.  
	
\paragraph{Acknowledgments}

	My deepest gratitude goes to Seraphina Lee for providing an early draft of her paper \cite{Lee} and a Sage program on which these computations are based, and for her continued interest and discussions relevant to this work, in particular for catching some errors in earlier drafts. I also thank Jonas Bergstr\"om for helpful discussions regarding the range in which the conjectures on the cohomology of local systems on \(\mathcal{A}_3\) are unconditional. 
	
	I would also like to thank Jim Arthur for his support, and Julia Gordon for giving a talk at the Fields Institute Beyond Endoscopy Mini-Conference which so clearly emphasized to me the connection between arithmetic statistics for Abelian varieties and results of Langlands and Kottwitz. 
	
	Finally I would like to thank Benson Farb and Dan Petersen for encouraging this work in the beginning, and Daniel Litt for encouraging me to finally finish it. 

\section{Arithmetic Statistics and Cohomology of Moduli Stacks}

	We now explain the method we use to study point counts of Abelian varieties over finite fields in terms of the \(\ell\)-adic cohomology of their moduli stacks, following Lee \cite{Lee}.
	
\paragraph{Moduli of Abelian Varieties}
	
	Let \(\mathcal{A}_g\) be the moduli stack of principally polarized Abelian varieties of dimension \(g\) which is a smooth Deligne-Mumford stack of dimension \(\mathrm{dim}(\mathcal{A}_g)=\frac{g(g+1)}{2}\) over \(\mathbb{Z}\) (and hence over any \(\mathbb{F}_q\) by base change) and let \(\mathcal{A}_g(\mathbb{F}_q)\) be the groupoid of principally polarized Abelian varieties of dimension \(g\) over \(\mathbb{F}_q\). Let \(\pi:\mathcal{X}_g\rightarrow\mathcal{A}_g\) be the universal family of Abelian varieties over \(\mathcal{A}_g\). For \(n\geq 1\) consider the \(n\)-th fiber power of the universal family
	\begin{align*}
		\pi^n:\mathcal{X}^{\times n}_g:=\underbrace{\mathcal{X}_g\times_{\mathcal{A}_g}\hdots\times_{\mathcal{A}_g}\mathcal{X}_g}_n\rightarrow\mathcal{A}_g
	\end{align*}
	which is a smooth Deligne-Mumford stack of dimension \(\mathrm{dim}(\mathcal{X}^{\times n}_g)=\frac{g(g+1)}{2}+ng\) over \(\mathbb{Z}\) (and hence over any \(\mathbb{F}_q\) by base change). The fiber of \(\pi^n:\mathcal{X}^{\times n}_g\rightarrow\mathcal{A}_g\) over a point \([A,\lambda]\in\mathcal{A}_g\) is the product \(A^n\), so the point counts \(\#\mathcal{X}^{\times n}_g(\mathbb{F}_q)\) encode the point counts \(\#A(\mathbb{F}_q)^n\) averaged over their moduli and weighted by the number of automorphisms. 
	
	By definition the expected value \(\mathbb{E}(\#A_g(\mathbb{F}_q)^n)\) of the random variable \(\#A_g(\mathbb{F}_q)^n\) with respect the probability measure \(\mu_{\mathcal{A}_g(\mathbb{F}_q)}\) defined in the introduction is given 
	\begin{align*}
		\mathbb{E}(\#A_g(\mathbb{F}_q)^n)=\sum_{[A,\lambda]\in[\mathcal{A}_g(\mathbb{F}_q)]}\frac{\#A(\mathbb{F}_q)^n}{\#\mathcal{A}_g(\mathbb{F}_q)\#\mathrm{Aut}_{\mathbb{F}_q}(A,\lambda)}
	\end{align*}
	which are related to the groupoid cardinality \(\#\mathcal{X}^{\times n}_g(\mathbb{F}_q)\) as follows: 
	
	\begin{proposition}
	(Compare to \cite[Lemma 6.8]{Lee}) The expected value \(\mathbb{E}(\#A_g(\mathbb{F}_q)^n)\) is given 
	\begin{align*}
		\mathbb{E}(\#A_g(\mathbb{F}_q)^n):=\frac{\#\mathcal{X}^{\times n}_g(\mathbb{F}_q)}{\#\mathcal{A}_g(\mathbb{F}_q)}
	\end{align*}
	\end{proposition}
	\begin{proof}
	Let \([A,\lambda]\in[\mathcal{A}_g(\mathbb{F}_q)]\) and consider the action of \(\mathrm{Aut}_{\mathbb{F}_q}(A,\lambda)\) on \(A^n\). Consider the action groupoid \([A(\mathbb{F}_q)^n]:=A(\mathbb{F}_q)^n/\!/\mathrm{Aut}_{\mathbb{F}_q}(A,\lambda)\). For \(\underline{x}\in A(\mathbb{F}_q)^n\) let \(\mathrm{Aut}_{\mathbb{F}_q}(A,\lambda;\underline{x})\subseteq\mathrm{Aut}_{\mathbb{F}_q}(A,\lambda)\) be the subgroup stabilizing \(\underline{x}\), and let \(\mathrm{Aut}_{\mathbb{F}_q}(A,\lambda)\cdot\underline{x}\) be the \(\mathrm{Aut}_{\mathbb{F}_q}(A,\lambda)\)-orbit of \(\underline{x}\). By the orbit-stabilizer theorem we have
	\begin{align*}
		\sum_{[\underline{x}]\in[A(\mathbb{F}_q)^n]}\frac{1}{\#\mathrm{Aut}(A,\lambda;\underline{x})}=\sum_{[\underline{x}]\in[A(\mathbb{F}_q)^n]}\frac{\#(\mathrm{Aut}_{\mathbb{F}_q}(A,\lambda)\cdot\underline{x})}{\#\mathrm{Aut}_{\mathbb{F}_q}(A,\lambda)}=\frac{\#A(\mathbb{F}_q)^n}{\#\mathrm{Aut}_{\mathbb{F}_q}(A,\lambda)}
	\end{align*}
	It follows that
	\begin{align*}
		\mathbb{E}(\#A_g(\mathbb{F}_q)^n) &= \sum_{[A,\lambda]\in[\mathcal{A}_g(\mathbb{F}_q)]}\frac{\#A(\mathbb{F}_q)^n}{\#\mathcal{A}_g(\mathbb{F}_q)\#\mathrm{Aut}_{\mathbb{F}_q}(A,\lambda)}\\
			&= \frac{1}{\#\mathcal{A}_g(\mathbb{F}_q)}\sum_{[A,\lambda]\in[\mathcal{A}_g(\mathbb{F}_q)]}\sum_{[\underline{x}]\in[A(\mathbb{F}_q)^n]}\frac{1}{\#\mathrm{Aut}_{\mathbb{F}_q}(A,\lambda;\underline{x})}\\
			&= \frac{1}{\#\mathcal{A}_g(\mathbb{F}_q)}\sum_{[A,\lambda;\underline{x}]\in[\mathcal{X}^{\times n}_g(\mathbb{F}_q)]}\frac{1}{\#\mathrm{Aut}_{\mathbb{F}_q}(A,\lambda;\underline{x})}=\frac{\#\mathcal{X}^{\times n}_g(\mathbb{F}_q)}{\#\mathcal{A}_g(\mathbb{F}_q)} \qedhere
	\end{align*}
	\end{proof}
	
	We will consider the moment generating function 
	\begin{align*}
		M_{\#A_g(\mathbb{F}_q)}(t):=\sum_{n\geq 0}\mathbb{E}(\#A_g(\mathbb{F}_q)^n)\frac{t^n}{n!}=\sum_{n\geq 0}\frac{\#\mathcal{X}^{\times n}_g(\mathbb{F}_q)}{\#\mathcal{A}_g(\mathbb{F}_q)}\frac{t^n}{n!}
	\end{align*}
	and we will consider the following normalization of the moment generating function
	\begin{align*}
		\widetilde{M}_{\#A_g(\mathbb{F}_q)}(t):=M_{\#A_g(\mathbb{F}_q)}(q^{-g}t)=\sum_{n\geq 0}q^{-ng}\frac{\#\mathcal{X}^{\times n}_g(\mathbb{F}_q)}{\#\mathcal{A}_g(\mathbb{F}_q)}\frac{t^n}{n!}
	\end{align*}
	which behaves better in the limit \(g\rightarrow\infty\). 
	
\paragraph{Grothendieck-Lefschetz Trace Formula}

	Now let \(\mathcal{X}\) be a Deligne-Mumford stack of finite type over \(\mathbb{F}_q\), and fix a prime \(\ell\) not dividing \(q\). For \(\mathbb{V}\) an \'etale \(\mathbb{Q}_\ell\)-sheaf on \(\mathcal{X}\) along with a choice of \(\mathbb{Z}_\ell\)-lattice \(\mathbb{V}_0\) write \(H^i(\mathcal{X},\mathbb{V})\) for the \(\ell\)-adic \'etale cohomology \(H^i_\mathrm{et}(\mathcal{X}_{\overline{\mathbb{F}}_q},\mathbb{V})=\varprojlim_nH^i_\mathrm{et}(\mathcal{X}_{\overline{\mathbb{F}}_q},\mathbb{V}_0/\ell^n)\otimes_{\mathbb{Z}_\ell}\mathbb{Q}_\ell\) and write \(\phi_q:H^i(\mathcal{X},\mathbb{V})\rightarrow H^i(\mathcal{X},\mathbb{V})\) for the arithmetic Frobenius. Similarly, write \(H^i_\mathrm{c}(\mathcal{X},\mathbb{V})\) for the compactly supported \(\ell\)-adic \'etale cohomology \(H^i_\mathrm{c}(\mathcal{X}_{\overline{\mathbb{F}}_q},\mathbb{V})=\varprojlim_nH^i_\mathrm{c,et}(\mathcal{X}_{\overline{\mathbb{F}}_q},\mathbb{V}_0/\ell^n)\otimes_{\mathbb{Z}_\ell}\mathbb{Q}_\ell\) and write \(\mathrm{Frob}_q:H^i(\mathcal{X},\mathbb{V})\rightarrow H^i(\mathcal{X},\mathbb{V})\) for the geometric Frobenius.
	
	When \(\mathcal{X}\) is smooth and has constant dimension the groupoid cardinality \(\#\mathcal{X}(\mathbb{F}_q)\) can be computed by a Grothendieck-Lefschetz trace formula as the alternating sum of traces of arithmetic (geometric) Frobenius on the (compactly supported) \(\ell\)-adic cohomology of \(\mathcal{X}\): 
	
	\begin{proposition}
	\label{GrothendieckLefschetz} Let \(\mathcal{X}\) be a smooth Deligne-Mumford stack of finite type and constant dimension \(d\) over \(\mathbb{F}_q\). Then we have
	\begin{align*}
		\#\mathcal{X}(\mathbb{F}_q)=q^d\sum_{i\geq 0}(-1)^i\mathrm{tr}(\phi_q|H^i(\mathcal{X},\mathbb{Q}_\ell))=\sum_{i\geq 0}(-1)^i\mathrm{tr}(\mathrm{Frob}_q|H^i_\mathrm{c}(\mathcal{X},\mathbb{Q}_\ell))
	\end{align*}
	\end{proposition}
	\begin{proof}
	The first equality follows by \cite[Theorem 2.4.5]{BehrendThesis}, noting that the \'etale cohomology of Deligne-Mumford stacks agrees with the smooth cohomology used in this theorem. The second equality follows by Poincare duality (see \cite[Proposition 2.30]{Zheng} for the case of Deligne-Mumford stacks), noting that \(q^d\mathrm{tr}(\phi_q|H^i(\mathcal{X},\mathbb{Q}_\ell))=\mathrm{tr}(\mathrm{Frob}_q|H^{2d-i}_\mathrm{c}(\mathcal{X},\mathbb{Q}_\ell))\). 
	\end{proof}
	
	It follows that we have
	\begin{align*}
		\mathbb{E}(\#A_g(\mathbb{F}_q)^n)=\frac{\mathrm{tr}(\mathrm{Frob}_q|e_\mathrm{c}(\mathcal{X}^{\times n}_g,\mathbb{Q}_\ell))}{\mathrm{tr}(\mathrm{Frob}_q|e_\mathrm{c}(\mathcal{A}_g,\mathbb{Q}_\ell))}:=\frac{\sum_{i\geq 0}(-1)^i\mathrm{tr}(\mathrm{Frob}_q|H^i_\mathrm{c}(\mathcal{X}^{\times n}_g,\mathbb{Q}_\ell))}{\sum_{i\geq 0}(-1)^i\mathrm{tr}(\mathrm{Frob}_q|H^i_\mathrm{c}(\mathcal{A}_g,\mathbb{Q}_\ell))}
	\end{align*}
	It remains to compute the Euler characteristics \(e(\mathcal{X}^{\times n}_g,\mathbb{Q}_\ell):=\sum_{i\geq 0}(-1)^iH^i(\mathcal{X}^{\times n}_g,\mathbb{Q}_\ell)\), or Poincare dually the compactly supported Euler characteristics \(e_\mathrm{c}(\mathcal{X}^{\times n}_g,\mathbb{Q}_\ell):=\sum_{i\geq 0}(-1)^iH^i_\mathrm{c}(\mathcal{X}^{\times n}_g,\mathbb{Q}_\ell)\), as elements of the Grothendieck group of \(\ell\)-adic Galois representations. 
	
\paragraph{Leray Spectral Sequence} 

	Now we would like to compute the cohomology of \(\mathcal{X}^{\times n}_g\) in terms of the cohomology of local systems on \(\mathcal{A}_g\). We observe that the Leray spectral sequence for the morphism \(\pi^n:\mathcal{X}^{\times n}_g\rightarrow\mathcal{A}_g\) degenerates at the \(E_2\)-page, as it does for smooth projective morphisms of schemes: 

	\begin{proposition}
	\label{Leray} (Compare to \cite[Proposition 2.8]{Lee}) We have a spectral sequence
	\begin{align*}
		E^{i,j}_2=H^i(\mathcal{A}_g,\mathbb{R}^j\pi^n_*\mathbb{Q}_\ell)\Rightarrow H^{i+j}(\mathcal{X}^{\times n}_g,\mathbb{Q}_\ell)
	\end{align*}
	which degenerates at the \(E_2\)-page, and we have a spectral sequence
	\begin{align*}
		E^{i,j}_2=H^i_\mathrm{c}(\mathcal{A}_g,\mathbb{R}^j\pi^n_*\mathbb{Q}_\ell)\Rightarrow H^{i+j}_\mathrm{c}(\mathcal{X}^{\times n}_g,\mathbb{Q}_\ell)
	\end{align*}
	which degenerates at the \(E_2\)-page.
	\end{proposition}
	\begin{proof}
	Let \(N\geq 3\) and let \(\mathcal{A}_g[N]\) be the moduli stack of principally polarized Abelian varieties of dimension \(g\) with full level \(N\) structure, which is a smooth quasi-projective scheme over \(\mathbb{Z}[\frac{1}{N}]\) (and hence over \(\mathbb{Q}\) or over any \(\mathbb{F}_q\) for \(q=p^k\) with \(p\nmid N\) by base change). Let \(\pi:\mathcal{X}_g[N]\rightarrow\mathcal{A}_g[N]\) be the universal family of Abelian varieties over \(\mathcal{A}_g[N]\). For \(n\geq 1\) consider the \(n\)-th fiber power of the universal family 
	\begin{align*}
		\pi^n:\mathcal{X}_g[N]^{\times n}:=\underbrace{\mathcal{X}_g[N]\times_{\mathcal{A}_g}\hdots\times_{\mathcal{A}_g}\mathcal{X}_g[N]}_n\rightarrow\mathcal{A}_g[N]
	\end{align*}
	which is a smooth quasi-projective scheme over \(\mathbb{Z}[\frac{1}{N}]\) (and hence over \(\mathbb{Q}\) or over any \(\mathbb{F}_q\) for \(q=p^k\) with \(p\nmid N\) by base change). Since \(\pi^n:\mathcal{X}_g[N]^{\times n}\rightarrow\mathcal{A}_g[N]\) is a smooth projective morphism, the Leray spectral sequence
	\begin{align*}
		E^{i,j}_2=H^i(\mathcal{A}_g[N],\mathbb{R}^j\pi^n_*\mathbb{Q}_\ell)\Rightarrow H^{i+j}(\mathcal{X}_g[N]^{\times n},\mathbb{Q}_\ell)
	\end{align*}
	degenerates at the \(E_2\)-page (see for example \cite[Proposition 2.4]{DeligneThesis} and \cite[Theorem 4.1.1]{DeligneWeilII}), so we have an isomorphism
	\begin{align*}
		\bigoplus_{i+j=k}H^i(\mathcal{A}_g[N],\mathbb{R}^j\pi^n_*\mathbb{Q}_\ell)\simeq H^k(\mathcal{X}_g[N]^{\times n},\mathbb{Q}_\ell)
	\end{align*}
	of \(\ell\)-adic Galois representations up to semisimplification. Now by the Hochschild-Serre spectral sequence \cite[Theorem 2.20]{Milne} for the \(\mathrm{Sp}_{2g}(\mathbb{Z}/N\mathbb{Z})\)-quotient \(\mathcal{A}_g[N]\rightarrow\mathcal{A}_g\) we have
	\begin{align*}
		H^i(\mathcal{A}_g[N],\mathbb{R}^j\pi^n_*\mathbb{Q}_\ell)^{\mathrm{Sp}_{2g}(\mathbb{Z}/N\mathbb{Z})}\simeq H^i(\mathcal{A}_g,\mathbb{R}^j\pi^n_*\mathbb{Q}_\ell)
	\end{align*}
	and by the Hochschild-Serre spectral sequence for the \(\mathrm{Sp}_{2g}(\mathbb{Z}/N\mathbb{Z})\)-quotient \(\mathcal{X}_g[N]^{\times n}\rightarrow\mathcal{X}^{\times n}_g\) (with \(\mathrm{Sp}_{2g}(\mathbb{Z}/N\mathbb{Z})\) acting diagonally) we have 
	\begin{align*}
		\bigoplus_{i+j=k}H^i(\mathcal{A}_g[N],\mathbb{R}^j\pi^n_*\mathbb{Q}_\ell)^{\mathrm{Sp}_{2g}(\mathbb{Z}/N\mathbb{Z})}\simeq H^k(\mathcal{X}_g[N]^{\times n},\mathbb{Q}_\ell)^{\mathrm{Sp}_{2g}(\mathbb{Z}/N\mathbb{Z})}\simeq H^k(\mathcal{X}^{\times n}_g,\mathbb{Q}_\ell)
	\end{align*}
	so by naturality of the Leray spectral sequence we can take \(\mathrm{Sp}_{2g}(\mathbb{Z}/N\mathbb{Z})\)-invariants and it follows that the Leray spectral sequence
	\begin{align*}
		E^{i,j}_2=H^i(\mathcal{A}_g,\mathbb{R}^j\pi^n_*\mathbb{Q}_\ell)\Rightarrow H^{i+j}(\mathcal{X}^{\times n}_g,\mathbb{Q}_\ell)
	\end{align*}
	degenerates at the \(E_2\)-page. The proof for the Leray spectral sequence for compactly supported cohomology is similar, and follows by Poincare duality, noting that \(\mathbb{R}^j\pi^n_!\mathbb{Q}_\ell\simeq\mathbb{R}^j\pi^n_*\mathbb{Q}_\ell\) since \(\pi^n\) is proper. 
	\end{proof}

	\begin{corollary}
	\label{EulerChar} We have
	\begin{align*}
		e(\mathcal{X}^{\times n}_g,\mathbb{Q}_\ell)=\sum_{j\geq 0}(-1)^je(\mathcal{A}_g,\mathbb{R}^j\pi^n_*\mathbb{Q}_\ell)
	\end{align*}
	and we have 
	\begin{align*}
		e_\mathrm{c}(\mathcal{X}^{\times n}_g,\mathbb{Q}_\ell)=\sum_{j\geq 0}(-1)^je_\mathrm{c}(\mathcal{A}_g,\mathbb{R}^j\pi^n_*\mathbb{Q}_\ell)
	\end{align*}
	as an element of the Grothendieck group of \(\ell\)-adic Galois representations.  
	\end{corollary}

\paragraph{K\"unneth Formula}

	We can make one further simplification by using the K\"unneth formula to express the \(\ell\)-adic sheaves \(\mathbb{R}^j\pi^n_*\mathbb{Q}_\ell\) in terms of the \(\ell\)-adic sheaves \(\mathbb{R}^j\pi_*\mathbb{Q}_\ell\): 
	
	\begin{proposition}
	\label{Kunneth} We have an isomorphism
	\begin{align*}
		\mathbb{R}^j\pi^n_*\mathbb{Q}_\ell\simeq\bigoplus_{\substack{\lambda\vdash j\\ \lambda=(1^{j_1}\hdots n^{j_n})}}\bigotimes_{1\leq i\leq n}\wedge^{j_i}\mathbb{V}
	\end{align*}
	where \(\mathbb{V}=\mathbb{R}^1\pi_*\mathbb{Q}_\ell\) is the \(\ell\)-adic local system on \(\mathcal{A}_g\) whose fiber over \([A,\lambda]\in\mathcal{A}_g\) is \(H^1(A,\mathbb{Q}_\ell)\) corresponding to the standard representation of \(\mathrm{Sp}_{2g}\). 
	\end{proposition}
	\begin{proof}
	By the K\"unneth formula (see \cite[Corollary 2.20]{Zheng} for the case of Deligne-Mumford stacks) we have have an isomorphism \(\mathbb{R}^j\pi^n_*\mathbb{Q}_\ell\simeq\bigoplus_{j_1+j_2=j}(\mathbb{R}^{j_1}\pi^{n-1}_*\mathbb{Q}_\ell)\otimes(\mathbb{R}^{j_2}\pi_*\mathbb{Q}_\ell)\), so by induction on \(n\) it follows that
	\begin{align*}
		\mathbb{R}^j\pi^n_*\mathbb{Q}_\ell\simeq\bigoplus_{\substack{\lambda\vdash j\\ \lambda=(1^{j_1}\hdots n^{j_n})}}\bigotimes_{1\leq i\leq n}\mathbb{R}^{j_i}\pi_*\mathbb{Q}_\ell
	\end{align*}
	Now the result follows since \(\mathbb{R}^j\pi_*\mathbb{Q}_\ell\simeq\wedge^j\mathbb{V}\) is the \(\ell\)-adic local sytem on \(\mathcal{A}_g\) whose fiber over \([A,\lambda]\in\mathcal{A}_g\) is \(H^j(A,\mathbb{Q}_\ell)\simeq\wedge^jH^1(A,\mathbb{Q}_\ell)\). 
	\end{proof}
	
	For \(\lambda=(\lambda_1\geq\hdots\geq\lambda_g\geq 0)\) a highest weight for \(\mathrm{Sp}_{2g}\) let \(\mathbb{V}_\lambda\) be the \(\ell\)-adic local system on \(\mathcal{A}_g\) occurring in \(\mathrm{Sym}^{\lambda_1-\lambda_2}(\mathbb{V})\otimes\hdots\otimes\mathrm{Sym}^{\lambda_{g-1}-\lambda_g}(\wedge^{g-1}\mathbb{V})\otimes\mathrm{Sym}^{\lambda_g}(\wedge^g\mathbb{V})\) corresponding to the irreducible highest weight representation \(V_\lambda\) of \(\mathrm{Sp}_{2g}\). The tensor product of highest weight representations decomposes as a direct sum of highest weight representations with multiplicities
	\begin{align*}
		\mathbb{V}_\lambda\otimes\mathbb{V}_{\lambda'}\simeq\bigoplus_{\lambda''}m_{\lambda,\lambda',\lambda''}\mathbb{V}_{\lambda''}
	\end{align*}
	where the multiplicities \(m_{\lambda,\lambda',\lambda''}\) can be computed in terms of Littlewood-Richardson coefficients and the image of the specialization morphism from the universal character ring (see \cite[Theorem 3.1]{Koike} and \cite[Section 2.2]{KoikeTerada}, though we will not use this description in later computations).
	
	It follows that we have a decomposition 
	\begin{align*}
		\mathbb{R}^j\pi^n_*\mathbb{Q}_\ell\simeq\bigoplus_\lambda\mathbb{V}_\lambda(\tfrac{|\lambda|-j}{2})^{\oplus m^{j,n}_\lambda}
	\end{align*}
	where the \(\mathbb{V}_\lambda\) are irreducible \(\ell\)-adic local systems on \(\mathcal{A}_g\) with multiplicity \(m^{j,n}_\lambda\geq 0\) determined by Newell-Littlewood numbers, and where \(|\lambda|=\lambda_1+\hdots+\lambda_g\). Then we have 
	\begin{align*}
		e_\mathrm{c}(\mathcal{X}^{\times n}_g,\mathbb{Q}_\ell)=\sum_{j\geq 0}(-1)^j\sum_\lambda m^{j,n}_\lambda e_\mathrm{c}(\mathcal{A}_g,\mathbb{V}_\lambda)(\tfrac{|\lambda|-j}{2})=\sum_\lambda f^n_\lambda(\mathbb{L})e_\mathrm{c}(\mathcal{A}_g,\mathbb{V}_\lambda)
	\end{align*}
	as elements of the Grothendieck group of \(\ell\)-adic Galois representations, where \(f^n_\lambda(\mathbb{L})=\sum_{j\geq 0}(-1)^jm^{j,n}_\lambda\mathbb{L}^{\frac{j-|\lambda|}{2}}\) is a polynomial in the Lefschetz motive \(\mathbb{L}=\mathbb{Q}_\ell(-1)\), in which case by applying the Grothendieck-Lefschetz trace formula we obtain
	\begin{align*}
		\mathbb{E}(\#A(\mathbb{F}_q)^n)=\frac{\sum_\lambda\mathrm{tr}(\mathrm{Frob}_q|f^n_\lambda(\mathbb{L})e_\mathrm{c}(\mathcal{A}_g,\mathbb{V}_\lambda))}{\mathrm{tr}(\mathrm{Frob}_q|e_\mathrm{c}(\mathcal{A}_g,\mathbb{Q}_\ell))}=\frac{\sum_\lambda f^n_\lambda(q)\mathrm{tr}(\mathrm{Frob}_q|e_\mathrm{c}(\mathcal{A}_g,\mathbb{V}_\lambda))}{\mathrm{tr}(\mathrm{Frob}_q|e_\mathrm{c}(\mathcal{A}_g,\mathbb{Q}_\ell))}
	\end{align*}
	We have reduced the problem of computing the moments \(E(\#A(\mathbb{F}_q)^n)\) to the problem of computing the multiplicities \(m^{j,n}_\lambda\), and to the problem of computing the Euler characteristics \(e_\mathrm{c}(\mathcal{A}_g,\mathbb{V}_\lambda)\) as elements of the Grothendieck group of \(\ell\)-adic Galois representations. The first problem is straightforward, although it is perhaps not so easy to produce clean expressions for these multiplicities except for small \(g\). The second problem is more difficult: explicit computations are only known for \(g=1\) by results of Eichler-Shimura, for \(g=2\) by results of Lee-Weintraub \cite{LeeWeintraub} and Petersen \cite{Petersen}, and for \(g=3\) by results of Hain \cite{Hain} and conjectures of Bergstro\"m-Faber-van der Geer \cite{BFvdG3}. We will summarize these computations at the end of the paper. 

\section{Conjectures on Point Counts as $g\rightarrow\infty$}

	We now consider the asymptotics of the distributions of the point counts \(\#A_g(\mathbb{F}_q)\) in the limit \(g\rightarrow\infty\). Following the strategy of \cite{AEKWB}, we pose the following conjecture: 
	
	\begin{conjecture}
	\label{Conjecture1} (compare to \cite[Conjecture 1]{AEKWB}) Let \(\lambda=1+\frac{1}{q}+\frac{1}{q(q-1)}=\frac{1}{1-q^{-1}}\). For all \(n\geq 1\) we have
	\begin{align*}
		\lim_{g\rightarrow\infty}q^{-ng}\mathbb{E}(\#A_g(\mathbb{F}_q)^n)=\lim_{g\rightarrow\infty}q^{-ng}\frac{\#\mathcal{X}^{\times n}_g(\mathbb{F}_q)}{\#\mathcal{A}_g(\mathbb{F}_q)}=\lambda^{\frac{n(n+1)}{2}}
	\end{align*}
	\end{conjecture}
	
	In other words, for \(P(\lambda)\) the distribution with moment generating function \(M_{P(\lambda)}(t)=\sum_{n\geq 0}\lambda^{\frac{n(n+1)}{2}}\frac{t^n}{n!}\), the conjecture predicts
	\begin{align*}
		\lim_{g\rightarrow\infty}\widetilde{M}_{\#A_g(\mathbb{F}_q)}(t)=M_{P(\lambda)}(t)
	\end{align*}
	so that the distributions of the normalized point counts \(q^{-g}\#A_g(\mathbb{F}_q)\) converge to the distribution \(P(\lambda)\) in the limit \(g\rightarrow\infty\).
	
	\begin{remark}
	Let \(\mathcal{M}_g\) be the moduli stack of genus \(g\) curves and let \(\mathcal{M}_{g,n}\) be the moduli stack of genus \(g\) curves with \(n\) marked points, which are smooth Deligne-Mumford stacks over \(\mathbb{Z}\) (and hence over any \(\mathbb{F}_q\) by base change). On the discrete probability space \(([\mathcal{M}_g(\mathbb{F}_q)],2^{[\mathcal{M}_g(\mathbb{F}_q)]},\mu_{\mathcal{M}_g(\mathbb{F}_q)})\) consider the random variable \(\#C_g:[\mathcal{M}_g(\mathbb{F}_q)]\rightarrow\mathbb{Z}\) assigning to \([C]\in[\mathcal{M}_g(\mathbb{F}_q)]\) the point count \(\#C(\mathbb{F}_q)\). With the above normalization and with the same \(\lambda\) as above, \cite[Conjecture 1]{AEKWB} reads
	\begin{align*}
		\lim_{g\rightarrow\infty}q^{-n}\mathbb{E}(\#C_g(\mathbb{F}_q)_n)=\lim_{g\rightarrow\infty}q^{-n}\frac{\#\mathcal{M}_{g,n}(\mathbb{F}_q)}{\#\mathcal{M}_g(\mathbb{F}_q)}=\lambda^n
	\end{align*}
	where \(X_n=X(X-1)\hdots(X-n+1)\) is the falling factorial. In other words, for \(\mathrm{Pois}(\lambda)\) the Poisson distribution with mean \(\lambda\) and with falling moment generating function \(\underline{M}_{\mathrm{Pois}(\lambda)}(t)=\sum_{n\geq 0}\lambda^n\frac{t^n}{n!}\), the conjecture predicts
	\begin{align*}
		\lim_{g\rightarrow\infty}\widetilde{\underline{M}}_{\#C_g(\mathbb{F}_q)}(t)=\underline{M}_{\mathrm{Pois}(\lambda)}(t)
	\end{align*}
	where \(\widetilde{\underline{M}}_{\#C_g(\mathbb{F}_q)}(t):=\underline{M}_{\#C_g(\mathbb{F}_q)}(q^{-1}t)=\sum_{n\geq 0}q^{-n}\frac{\#\mathcal{M}_{g,n}(\mathbb{F}_q)}{\#\mathcal{M}_g(\mathbb{F}_q)}\frac{t^n}{n!}\) is the normalization of the falling moment generating function \(\underline{M}_{\#C_g(\mathbb{F}_q)}(t):=\sum_{n\geq 0}\mathbb{E}(\#C_g(\mathbb{F}_q)_n)\frac{t^n}{n!}=\sum_{n\geq 0}\frac{\#\mathcal{M}_{g,n}(\mathbb{F}_q)}{\#\mathcal{M}_g(\mathbb{F}_q)}\frac{t^n}{n!}\), so that the distributions of the normalized point counts \(q^{-1}\#C_g(\mathbb{F}_q)\) converge to a Poisson distribution with mean \(\lambda\) in the limit \(g\rightarrow\infty\). It would be interesting to give a conceptual explanation for why the same \(\lambda\) appears. 
	\end{remark}
	
\paragraph{Cohomological Stability}
	
	We now review some results on cohomological stability for \(\mathcal{A}_g\) and \(\mathcal{X}^{\times n}_g\). Consider the product morphism 
	\begin{align*}
		\mathcal{A}_{g_1}(\mathbb{C})\times\mathcal{A}_{g_2}(\mathbb{C}) &\rightarrow \mathcal{A}_{g_1+g_2}(\mathbb{C})\\
		([A_1],[A_2]) &\mapsto [A_1\times A_2]
	\end{align*}
	Choosing an elliptic curve \([E]\in\mathcal{A}_1(\mathbb{C})\) we obtain a morphism 
	\begin{align*}
		\mathcal{A}_g(\mathbb{C}) &\rightarrow \mathcal{A}_{g+1}(\mathbb{C})\\
		[A] &\mapsto [A\times E]
	\end{align*}
	such that induced morphism on cohomology \(H^*(\mathcal{A}_{g+1}(\mathbb{C}),\mathbb{Q})\rightarrow H^*(\mathcal{A}_g(\mathbb{C}),\mathbb{Q})\) does not depend on the choice of elliptic curve \(E\), since any two elliptic curves over \(\mathbb{C}\) are homotopy equivalent. Similarly we obtain a morphism 
	\begin{align*}
		\mathcal{X}^{\times n}_g(\mathbb{C}) &\rightarrow \mathcal{X}^{\times n}_{g+1}(\mathbb{C})\\
		[A;x_1,\hdots,x_n] &\mapsto [A\times E;(x_1,0),\hdots,(x_n,0)]
	\end{align*}
	such that the induced morphism on cohomology \(H^*(\mathcal{X}^{\times n}_{g+1}(\mathbb{C}),\mathbb{Q})\rightarrow H^*(\mathcal{X}^{\times n}_g(\mathbb{C}),\mathbb{Q})\) does not depend on the choice of elliptic curve \(E\) for the same reason as above. 
	
	By \cite[Theorem 7.5]{BorelI} and \cite[Theorem 4.4]{BorelII} (and by \cite[Theorem 3.2]{Hain1} making the stability range explicit), the cohomology \(H^i(\mathcal{A}_g(\mathbb{C}),\mathbb{Q})\) stabilizes in degrees \(0\leq i\leq g-1\), where it agrees with the inverse limit \(H^i(\mathcal{A}_\infty(\mathbb{C}),\mathbb{Q})=\varprojlim_gH^*(\mathcal{A}_g(\mathbb{C}),\mathbb{Q})\). The stable cohomology \(H^*(\mathcal{A}_\infty(\mathbb{C}),\mathbb{Q})\) is a free graded \(\mathbb{Q}\)-algebra, which has the following description. 
	
	Consider the graded \(\mathbb{Q}\)-algebra \(S^*=\mathbb{Q}[\lambda_i]_{i\geq 1\text{ odd}}\) where \(\mathrm{deg}(\lambda_i)=2i\). We have an isomorphism of graded \(\mathbb{Q}\)-algebras 
	\begin{align*}
		S^* &\xrightarrow{\sim} H^*(\mathcal{A}_\infty(\mathbb{C}),\mathbb{Q})\\
			\lambda_i &\mapsto \pi_*u_i
	\end{align*}
	where \(u_i=c_i(\Omega_{\mathcal{X}_g/\mathcal{A}_g})\) is the \(i\)-th Chern class of the relative canonical bundle of the universal family \(\pi:\mathcal{X}_g\rightarrow\mathcal{A}_g\). In particular we have an isomorphism \(S^i\xrightarrow{\sim}H^i(\mathcal{A}_g(\mathbb{C}),\mathbb{Q})\) for all \(0\leq i\leq g-1\). 
	
	More generally by \cite[Theorem 6.1]{GrushevskyHulekTommasi} the cohomology \(H^i(\mathcal{X}^{\times n}_g(\mathbb{C}),\mathbb{Q})\) stabilizes in degrees \(0\leq i\leq g-1\), where it agrees with the inverse limit \(H^i(\mathcal{X}^{\times n}_\infty(\mathbb{C}),\mathbb{Q})=\varprojlim_gH^i(\mathcal{X}^{\times n}_g(\mathbb{C}),\mathbb{Q})\). The stable cohomology \(H^*(\mathcal{X}^{\times n}_\infty(\mathbb{C}),\mathbb{Q})\) is a free \(H^*(\mathcal{A}_\infty(\mathbb{C}),\mathbb{Q})\)-algebra, which has the following description.
	
	Consider the graded \(\mathbb{Q}\)-algebra \(S^*_n=S^*[T_i]_{1\leq i\leq n}[P_{i,j}]_{1\leq i<j\leq n}\) where \(\mathrm{deg}(T_i)=\mathrm{deg}(P_{i,j})=2\). We have an isomorphism of graded \(S^*\simeq H^*(\mathcal{A}_\infty(\mathbb{C}),\mathbb{Q})\)-algebras 
	\begin{align*}
		S^*_n &\xrightarrow{\sim} H^*(\mathcal{X}^{\times n}_\infty(\mathbb{C}),\mathbb{Q})\\
			\lambda_i &\mapsto \pi_*u_i\\
			T_i &\mapsto \pi^*_i\Theta\\
			P_{i,j} &\mapsto \pi^*_{i,j}P
	\end{align*}
	where \(\Theta\in H^2(\mathcal{X}_g(\mathbb{C}),\mathbb{Q})\) is the class of the universal theta divisor trivialized along the zero section and \(\pi_i:\mathcal{X}^{\times n}_g\rightarrow\mathcal{X}_g\) is the \(i\)-th projection, and where \(P\in H^2(\mathcal{X}^{\times 2}_g(\mathbb{C}),\mathbb{Q})\) is the class of the universal Poincare divisor trivialized along the zero section and \(\pi_{i,j}:\mathcal{X}^{\times n}_g\rightarrow\mathcal{X}^{\times 2}_g\) is the \((i,j)\)-th projection. In particular we have an isomorphism \(S^i_n\xrightarrow{\sim}H^i(\mathcal{X}^{\times n}_g(\mathbb{C}),\mathbb{Q})\) for all \(0\leq i\leq g-1\). 
	
	We now consider the action of Frobenius on \(\ell\)-adic cohomology. Consider the graded \(\mathbb{Q}_\ell\)-algebra \(S^*_{n,\ell}=S^*_n\otimes_\mathbb{Q}\mathbb{Q}_\ell\) with endomorphism \(\mathrm{Frob}_q\) given by \(\mathrm{Frob}_q(\lambda_i)=q^i\lambda_i\) and \(\mathrm{Frob}_q(T_i)=qT_i\), and \(\mathrm{Frob}_q(P_{i,j})=qP_{i,j}\). We have a morphism of graded \(\mathbb{Q}_\ell\)-algebras \(S^*_{n,\ell}\rightarrow H^*(\mathcal{X}^{\times n}_{g,\overline{\mathbb{F}}_q},\mathbb{Q}_\ell)\) defined the same way as the morphism of graded \(\mathbb{Q}\)-algebras \(S^*_n\rightarrow H^*(\mathcal{X}^{\times n}_g(\mathbb{C}),\mathbb{Q})\) obtained from the above construction. The stable classes \(\pi_*u_i\) and \(\pi^*_i\Theta\) and \(\pi^*_{i,j}P\) are Tate type since they are formed through pullback and pushforward of Chern classes, in particular the above morphism is \(\mathrm{Frob}_q\)-equivariant. 
	
	\begin{proposition}
	A choice of embedding \(\overline{\mathbb{Q}}_p\hookrightarrow\mathbb{C}\) induces a sequence of functorial isomorphisms
	\begin{align*}
		H^i(\mathcal{X}^{\times n}_{g,\overline{\mathbb{F}}_q},\mathbb{Q}_\ell)\xrightarrow{\sim}H^i(\mathcal{X}^{\times n}_{g,\mathbb{C}},\mathbb{Q}_\ell)\xrightarrow{\sim}H^i(\mathcal{X}^{\times n}_g(\mathbb{C}),\mathbb{Q}_\ell)
	\end{align*}
	under which the classes \(\pi^*u_i\) and \(\pi^*_i\Theta\) and \(\pi^*_{i,j}P\) map to the same classes by functoriality. 
	\end{proposition}
	\begin{proof}
	We employ \cite[Lemma 8]{AEKWB}: Let \(\overline{X}\) be a smooth proper scheme over \(\mathbb{Z}_p\), let \(D\) be a relative normal crossings divisor on \(\overline{X}\), let \(G\) be a finite group acting on \(\overline{X}\) and on \(D\), let \(X=\overline{X}-D\), and let \(\mathcal{X}=[X/G]\) be the corresponding stack quotient. Then a choice of embedding \(\overline{\mathbb{Q}}_p\hookrightarrow\mathbb{C}\) induces a sequence of functorial isomorphisms \(H^i(\mathcal{X}_{\overline{\mathbb{F}}_q},\mathbb{Q}_\ell)\xrightarrow{\sim}H^i(\mathcal{X}_\mathbb{C},\mathbb{Q}_\ell)\xrightarrow{\sim}H^i(\mathcal{X}(\mathbb{C}),\mathbb{Q}_\ell)\). 
	
	Now let \(N\geq 3\) and for \(n\geq 1\) consider the \(n\)-th fiber power of the universal family \(X=\mathcal{X}_g[N]^{\times n}\) over \(\mathcal{A}_g[N]\) which is a smooth quasi-projective scheme over \(\mathbb{Z}_p\) for \(p\nmid N\). Consider the toroidal compactification \(\overline{X}=(\mathcal{X}_g[N]^{\times n})^\mathrm{tor}\): by \cite[Chapter VI, Theorem 1.1]{FaltingsChai} (or more generally by \cite[Theorem 2.15(1)]{LanToroidal}) this is a smooth projective algebraic space over \(\mathbb{Z}_p\) for \(p\nmid N\) such that the complement \(D=\overline{X}-X\) is a relative (simple) normal crossings divisor. The natural action of the finite group \(G=\mathrm{Sp}_{2g}(\mathbb{Z}/N\mathbb{Z})\) on \(X\) extends to an action on \(\overline{X}\) and on \(D\), and the corresponding stack quotient is given \(\mathcal{X}=[X/G]=\mathcal{X}^{\times n}_g\). Now the result follows, noting that \cite[Lemma 8]{AEKWB} still applies for algebraic spaces (when \(G\) is trivial the first isomorphism in the lemma follows from \cite[Proposition 4.3]{Nakayama} and the second isomorphism in the lemma follows from the comparison isomorphism \cite[Theorem I.11.6]{FreitagKiehl}, and in general the lemma follows from the Hochschild-Serre spectral sequence \cite[Theorem 2.20]{Milne}, and all of these still apply for algebraic spaces). 
	\end{proof}
	
	By composition with the morphism of graded \(\mathbb{Q}_\ell\)-algebras \(S^*_{n,\ell}\rightarrow H^*(\mathcal{X}^{\times n}_{g,\overline{\mathbb{F}}_q},\mathbb{Q}_\ell)\) we obtain an isomorphism \(S^i_{n,\ell}\xrightarrow{\sim}H^i(\mathcal{X}^{\times n}_g(\mathbb{C}),\mathbb{Q}_\ell)\) for all \(0\leq i\leq g-1\), obtained by tensoring the isomorphism \(S^i_n\xrightarrow{\sim}H^i(\mathcal{X}^{\times n}_g(\mathbb{C}),\mathbb{Q})\) over \(\mathbb{Q}\) with \(\mathbb{Q}_\ell\), in particular this does not depend on the choice of embedding \(\overline{\mathbb{Q}}_p\hookrightarrow\mathbb{C}\). It follows that we have an isomorphism \(S^*_{n,\ell}\xrightarrow{\sim}H^*(\mathcal{X}^{\times n}_{g,\overline{\mathbb{F}}_q},\mathbb{Q}_\ell)\) for all \(0\leq i\leq g-1\). In particular for \(0\leq i\leq g-1\) odd we have \(H^{2\mathrm{dim}(\mathcal{X}^{\times n}_g)-i}_\mathrm{c}(\mathcal{X}^{\times n}_{g,\overline{\mathbb{F}}_q},\mathbb{Q}_\ell)=0\), and for \(0\leq i\leq g-1\) even we have \(H^{2\mathrm{dim}(\mathcal{X}^{\times n})-i}_\mathrm{c}(\mathcal{X}^{\times n}_{g,\overline{\mathbb{F}}_q},\mathbb{Q}_\ell)=\mathrm{dim}_{\mathbb{Q}_\ell}(S^i_{n,\ell})\mathbb{L}^{\mathrm{dim}(\mathcal{X}^{\times n}_g)-\frac{i}{2}}\), by Poincare duality. 
	
\paragraph{Negligible Contributions to Point Counts as $g\rightarrow\infty$}
	
	Let \(R^*_\mathrm{c}(\mathcal{X}^{\times n}_{g,\overline{\mathbb{F}}_q},\mathbb{Q}_\ell)\) be the subring of \(H^*_\mathrm{c}(\mathcal{X}^{\times n}_{g,\overline{\mathbb{F}}_q},\mathbb{Q}_\ell)\) generated by the image of \(S^*_{n,\ell}\), and let \(B^*_\mathrm{c}(\mathcal{X}^{\times n}_{g,\overline{\mathbb{F}}_q},\mathbb{Q}_\ell)=H^*_\mathrm{c}(\mathcal{X}^{\times n}_{g,\overline{\mathbb{F}}_q},\mathbb{Q}_\ell)/R^*_\mathrm{c}(\mathcal{X}^{\times n}_{g,\overline{\mathbb{F}}_q},\mathbb{Q}_\ell)\). We conjecture that the traces of Frobenius on the classes not in the image of \(S^*_{n,\ell}\) should be negligible in the limit \(g\rightarrow\infty\): 
	
	\begin{conjecture}
	\label{Conjecture2} (compare to \cite[Heuristic 2]{AEKWB}) For all \(n\geq 0\) we have 
	\begin{align*}
		\lim_{g\rightarrow\infty}q^{-\mathrm{dim}(\mathcal{X}^{\times n}_g)}\sum_{0\leq i\leq 2\mathrm{dim}(\mathcal{X}^{\times n}_g)-g}(-1)^i\mathrm{tr}(\mathrm{Frob}_q|B^i_\mathrm{c}(\mathcal{X}^{\times n}_{g,\overline{\mathbb{F}}_q},\mathbb{Q}_\ell))=0
	\end{align*}
	\end{conjecture}
	
	We now show that \ref{Conjecture2} implies \ref{Conjecture1}, following the same strategy as in \cite[Theorem 3]{AEKWB}. We first review some results on cohomological stability, following \cite[Section 7]{HulekTommasi}.
	
	Now we break up the point count \(\#\mathcal{X}^{\times n}_g(\mathbb{F}_q)\) into stable, unstable, and negligible contributions: 
	\begin{align*}
		T^\mathrm{stable}_{g,n,q} &:= \sum_{0\leq i\leq g-1}(-1)^i\mathrm{tr}(\mathrm{Frob}_q|H^{2\mathrm{dim}(\mathcal{X}^{\times n}_g)-i}_\mathrm{c}(\mathcal{X}^{\times n}_{g,\overline{\mathbb{F}}_q},\mathbb{Q}_\ell))\\
			&= \sum_{0\leq i\leq g-1}(-1)^i\mathrm{tr}(\mathrm{Frob}_q|R^{2\mathrm{dim}(\mathcal{X}^{\times n}_g)-i}_\mathrm{c}(\mathcal{X}^{\times n}_{g,\overline{\mathbb{F}}_q},\mathbb{Q}_\ell))\\
		T^\mathrm{unstable}_{g,n,q} &:= \sum_{g\leq i\leq 2\mathrm{dim}(\mathcal{X}^{\times n}_g)}(-1)^i\mathrm{tr}(\mathrm{Frob}_q|R^{2\mathrm{dim}(\mathcal{X}^{\times n}_g)-i}_\mathrm{c}(\mathcal{X}^{\times n}_{g,\overline{\mathbb{F}}_q},\mathbb{Q}_\ell))\\
		N_{g,n,q} &:= \sum_{g\leq i\leq 2\mathrm{dim}(\mathcal{X}^{\times n}_g)}(-1)^i\mathrm{tr}(\mathrm{Frob}_q|B^{2\mathrm{dim}(\mathcal{X}^{\times n}_g)-i}_\mathrm{c}(\mathcal{X}^{\times n}_{g,\overline{\mathbb{F}}_q},\mathbb{Q}_\ell))
	\end{align*}
	Then by definition we have
	\begin{align*}
		\#\mathcal{X}^{\times n}_g(\mathbb{F}_q)=T^\mathrm{stable}_{g,n,q}+T^\mathrm{unstable}_{g,n,q}+N_{g,n,q}
	\end{align*}
	and the second conjecture is equivalent to the assertion that 
	\begin{align*}
		\lim_{g\rightarrow\infty}q^{-\mathrm{dim}(\mathcal{X}^{\times n}_g)}N_{g,n,q}=0
	\end{align*}
	for all \(n\geq 0\). Consider the Hilbert-Poincare series 
	\begin{align*}
		\mathrm{HS}_{S^*_n}(z):=\sum_{i\geq 0}\mathrm{dim}_\mathbb{Q}(S^i_n)z^i=\prod_{1\leq i\leq n}\frac{1}{1-z^2}\prod_{1\leq i<j\leq n}\frac{1}{1-z^2}\prod_{i\geq 1\text{ odd}}\frac{1}{1-z^{2i}}
	\end{align*}
	Now since \(R^i_{n,\ell}=R^i_n\otimes_\mathbb{Q}\mathbb{Q}_\ell\simeq R^{2\mathrm{dim}(\mathcal{X}^{\times n}_g)-i}_\mathrm{c}(\mathcal{X}^{\times n}_{g,\overline{\mathbb{F}}_q},\mathbb{Q}_\ell)\) we have
	\begin{align*}
		\lim_{g\rightarrow\infty}q^{-\mathrm{dim}(\mathcal{X}^{\times n}_g)}T^\mathrm{stable}_{g,n,q} &= \lim_{g\rightarrow\infty}q^{-\mathrm{dim}(\mathcal{X}^{\times n}_g)}\sum_{0\leq i\leq g-1}(-1)^i\mathrm{tr}(\mathrm{Frob}_q|R^{2\mathrm{dim}(\mathcal{X}^{\times n}_g)-i}_\mathrm{c}(\mathcal{X}^{\times n}_{g,\overline{\mathbb{F}}_q},\mathbb{Q}_\ell))\\
			&= \lim_{g\rightarrow\infty}q^{-\mathrm{dim}(\mathcal{X}^{\times n}_g)}\sum_{0\leq i\leq g-1}(-1)^i\mathrm{tr}(\mathrm{Frob}_q|S^i_{n,\ell})\\
			&= \lim_{g\rightarrow\infty}q^{-\mathrm{dim}(\mathcal{X}^{\times n}_g)}\sum_{0\leq i\leq g-1}q^{-i}\mathrm{dim}_\mathbb{Q}(S^{2i}_n)\\
			&= \sum_{i\geq 0}q^{-i}\mathrm{dim}_\mathbb{Q}(S^{2i}_n)=\mathrm{HS}_{S^*_n}(q^{-\frac{1}{2}})
	\end{align*}
	Let \(P_\mathrm{odd}(z)=\sum_{i\geq 0}p_\mathrm{odd}(i)z^i\) be the generating function for the odd partition numbers \(p_\mathrm{odd}(i)\) (the number of partitions of \(\{1,\hdots,i\}\) into odd parts), and let \(Q_n(z)=\sum_{i\geq 0}\binom{n+i-1}{i}z^i\) be the generating function for the binomial coefficients \(\binom{n+i-1}{i}\) (the number of multisets with cardinality \(n\) and weighted cardinality \(i\)). Then we have \(\mathrm{HS}_{S^*_n}(z)=Q_{\frac{n(n+1)}{2}}(z^2)P_\mathrm{odd}(z^2)\). 
	
	For the partition numbers \(p(i)\) (the number of partitions of \(\{1,\hdots,n\}\)) one has the exponential bound \(p_\mathrm{odd}(i)\leq p(i)\leq\exp(c\sqrt{i})\) for some constant \(c\) not depending on \(i\). In particular we have
	\begin{align*}
		\mathrm{dim}_\mathbb{Q}(S^{2i}_n)=\sum_{0\leq j\leq i}\binom{\frac{n(n+1)}{2}+j-i}{j-1}p_\mathrm{odd}(i-j)\leq\exp(c_n\sqrt{i})
	\end{align*}
	for some constant \(c_n\) not depending on \(i\). Since \(R^*_\mathrm{c}(\mathcal{X}^{\times n}_g,\mathbb{Q}_\ell)\) is defined in terms of the image of a morphism from \(S^*_n\) to cohomology we have \(\mathrm{dim}_{\mathbb{Q}_\ell}(R^{2\mathrm{dim}(\mathcal{X}^{\times n}_g)-2i}_\mathrm{c}(\mathcal{X}^{\times n}_{g,\overline{\mathbb{F}}_q},\mathbb{Q}_\ell))\leq\mathrm{dim}_\mathbb{Q}(S^{2i}_n)\), in particular we have \(\mathrm{dim}_{\mathbb{Q}_\ell}(R^{2\mathrm{dim}(\mathcal{X}^{\times n}_g)-2i}_\mathrm{c}(\mathcal{X}^{\times n}_{g,\overline{\mathbb{F}}_q},\mathbb{Q}_\ell))\leq\exp(c_n\sqrt{i})\). Now we have
	\begin{align*}
		\lim_{g\rightarrow\infty}q^{-\mathrm{dim}(\mathcal{X}^{\times n}_g)}T^\mathrm{unstable}_{g,n,q} &= \lim_{g\rightarrow\infty}q^{-\mathrm{dim}(\mathcal{X}^{\times n}_g)}\sum_{g\leq i\leq 2\mathrm{dim}(\mathcal{X}^{\times n}_g)}(-1)^i\mathrm{tr}(\mathrm{Frob}_q|R^{2\mathrm{dim}(\mathcal{X}^{\times n}_g)-i}_\mathrm{c}(\mathcal{X}^{\times n}_{g,\overline{\mathbb{F}}_q},\mathbb{Q}_\ell))\\
			&\leq \lim_{g\rightarrow\infty}\sum_{g\leq i\leq 2\mathrm{dim}(\mathcal{X}^{\times n}_g)}(-1)^iq^{-\frac{i}{2}}\mathrm{dim}_\mathbb{Q}(S^i_n)\\
			&\leq \lim_{g\rightarrow\infty}\sum_{g\leq i\leq 2\mathrm{dim}(\mathcal{X}^{\times n}_g)}(-1)^iq^{-\frac{i}{2}}\exp(c_n\sqrt{i})=0
	\end{align*}
	Now suppose that \(\lim_{g\rightarrow\infty}q^{-\mathrm{dim}(\mathcal{X}^{\times n}_g)}N_{g,n,q}=0\). Then we have 
	\begin{align*}
		\lim_{g\rightarrow\infty}q^{-ng}\frac{\#\mathcal{X}^{\times n}_g(\mathbb{F}_q)}{\#\mathcal{A}_g(\mathbb{F}_q)} &= \lim_{g\rightarrow\infty}q^{-ng}\frac{T^\mathrm{stable}_{g,n,q}+T^\mathrm{unstable}_{g,n,q}+N_{g,n,q}}{T^\mathrm{stable}_{g,0,q}+T^\mathrm{unstable}_{g,0,q}+N_{g,0,q}}\\
			&= \frac{\mathrm{HS}_{S^*_n}(q^{-\frac{1}{2}})}{\mathrm{HS}_{S^*}(q^{-\frac{1}{2}})}=\prod_{1\leq i\leq n}\frac{1}{1-q^{-1}}\prod_{1\leq i<j\leq n}\frac{1}{1-q^{-1}}=\lambda^{\frac{n(n+1)}{2}}
	\end{align*}
	so it follows that the second conjecture \ref{Conjecture2} on the negligible contribution of non-Tate classes to point counts implies the first conjecture \ref{Conjecture1} on the asymptotics of the distribution.

	Expanded around \(q=\infty\), the conjecture \ref{Conjecture1} predicts the following leading terms for the expected values \(\mathbb{E}(\#A(\mathbb{F}_q)^n)\) in the limit \(g\rightarrow\infty\): 
	\begin{align*}
		\lim_{g\rightarrow\infty}q^{-g}\mathbb{E}(\#A_g(\mathbb{F}_q)) &= 1+q^{-1}+q^{-2}+q^{-3}+q^{-4}+\hdots\\
		\lim_{g\rightarrow\infty}q^{-2g}\mathbb{E}(\#A_g(\mathbb{F}_q)^2) &= 1+3q^{-1}+6q^{-2}+10q^{-3}+15q^{-4}+\hdots\\
		\lim_{g\rightarrow\infty}q^{-3g}\mathbb{E}(\#A_g(\mathbb{F}_q)^3) &= 1+6q^{-1}+21q^{-2}+56q^{-3}+126q^{-4}+\hdots\\
		\lim_{g\rightarrow\infty}q^{-4g}\mathbb{E}(\#A_g(\mathbb{F}_q)^4) &= 1+10q^{-1}+55q^{-2}+220q^{-3}+715q^{-4}+\hdots\\
		\lim_{g\rightarrow\infty}q^{-5g}\mathbb{E}(\#A_g(\mathbb{F}_q)^5) &= 1+15q^{-1}+120q^{-2}+680q^{-3}+3060q^{-4}+\hdots
	\end{align*}

\section{Computations for $g=1$}

	Let \(\mathcal{A}_1\) be the moduli stack of elliptic curves, which is a smooth Deligne-Mumford stack of dimension \(1\) over \(\mathbb{Z}\). Let \(\pi:\mathcal{X}_1\rightarrow\mathcal{A}_1\) be the universal elliptic curve over \(\mathcal{A}_1\) and let \(\mathbb{V}=\mathbb{R}^1\pi_*\mathbb{Q}_\ell\) be the \(\ell\)-adic local system on \(\mathcal{A}_1\) corresponding to the standard representation of \(\mathrm{SL}_2\). For \(\lambda\geq 0\) an integer let \(\mathbb{V}_\lambda=\mathrm{Sym}^\lambda(\mathbb{V})\) be the \(\ell\)-adic local system on \(\mathcal{A}_1\) corresponding to the irreducible \(\lambda+1\)-dimensional representation of \(\mathrm{SL}_2\). For \(\lambda\) odd we have \(H^*(\mathcal{A}_1,\mathbb{V}_\lambda)=0\) since \(-\mathrm{id}\in\mathrm{SL}_2(\mathbb{Z})\) acts by multiplication by \((-1)^\lambda\) on the stalks of \(\mathbb{V}_\lambda\). 
	
	Let \(\mathbb{S}_{\Gamma(1)}[\lambda+2]=\bigoplus_f\rho_f\) be the \(\ell\)-adic Galois representation corresponding to cusp forms of weight \(\lambda+2\) for \(\Gamma(1)=\mathrm{SL}_2(\mathbb{Z})\): for each eigenform \(f\in S_{\lambda+2}(\Gamma(1))\) we have a \(2\)-dimensional \(\ell\)-adic Galois representation \(\rho_f\), and we have 
	\begin{align*}
		\mathrm{tr}(\mathrm{Frob}_p|\mathbb{S}_{\Gamma(1)}[\lambda+2])=\mathrm{tr}(T_p|S_{\lambda+2}(\Gamma(1)))
	\end{align*} 
	for every prime \(p\), which determines \(\mathbb{S}_{\Gamma(1)}[\lambda+1]\) as an element of the Grothendieck group of \(\ell\)-adic Galois representations. The \(\ell\)-adic Galois representation \(\rho_f\) is irreducible as a representation of \(\mathrm{Gal}(\overline{\mathbb{Q}}/\mathbb{Q})\) and of \(\mathrm{Gal}(\overline{\mathbb{F}}_p/\mathbb{F}_p)\). 
	
	By work of Eichler-Shimura and Deligne we have the following: 
	
	\begin{proposition}
	\label{EichlerShimura1} \cite[Theorem 2.3]{BFvdG3} For \(\lambda>0\) even we have 
	\begin{align*}
		e_\mathrm{c}(\mathcal{A}_1,\mathbb{V}_\lambda)=-H^1_\mathrm{c}(\mathcal{A}_1,\mathbb{V}_\lambda)=-\mathbb{S}_{\Gamma(1)}[\lambda+2]-1
	\end{align*}
	as an element of the Grothendieck group of \(\ell\)-adic Galois representations. 
	\end{proposition}
	
	This remains true for \(\lambda=0\) if we set \(\mathbb{S}_{\Gamma(1)}[2]:=-\mathbb{L}-1\): we have
	\begin{align*}
		e_\mathrm{c}(\mathcal{A}_1,\mathbb{Q}_\ell)=H^2_\mathrm{c}(\mathcal{A}_1,\mathbb{Q}_\ell)=\mathbb{L}
	\end{align*}	
	We will use the following values for the Euler characteristics \(e_\mathrm{c}(\mathcal{A}_1,\mathbb{V}_\lambda)\), which are obtained by combining \ref{EichlerShimura1} with the vanishing of the spaces \(S_{\lambda+2}(\Gamma(1))\) for all \(\lambda\geq 0\) with \(\lambda\leq 9\): 
	\begin{center}\begin{tabular}{|c|c|}
		\hline
		\(\lambda\) & \(e_\mathrm{c}(\mathcal{A}_1,\mathbb{V}_\lambda)\)\\
		\hline\hline
		\(0\) & \(\mathbb{L}\)\\
		\hline
		\(2\) & \(-1\)\\
		\hline
		\(4\) & \(-1\)\\
		\hline
	\end{tabular}\qquad\begin{tabular}{|c|c|}
		\hline
		\(\lambda\) & \(e_\mathrm{c}(\mathcal{A}_1,\mathbb{V}_\lambda)\)\\
		\hline\hline
		\(6\) & \(-1\)\\
		\hline
		\(8\) & \(-1\)\\
		\hline
		\(10\) & \(-\mathbb{S}_{\Gamma(1)}[12]-1\)\\
		\hline
	\end{tabular}\end{center}
	The space \(S_{12}(\Gamma(1))\) is spanned by the discriminant cusp form
	\begin{align*}
		\Delta=\sum_{n\geq 1}\tau(n)q^n=q-24q^2+252q^3-1472q^4+\hdots
	\end{align*}
	which contributes an irreducible \(2\)-dimensional \(\ell\)-adic Galois representation \(\mathbb{S}_{\Gamma(1)}[12]\) to \(H^1(\mathcal{A}_1,\mathbb{V}_{10})\), with the property that \(\mathrm{tr}(\mathrm{Frob}_p|\mathbb{S}_{\Gamma(1)}[12])=\tau(p)\), which is not polynomial in \(p\).
	
	
	We obtain the following result (compare to the tables at the end of \cite{Getzler}):

	\begin{theorem}
	\label{Theorem1} The cohomology \(H^i(\mathcal{X}^{\times n}_1,\mathbb{Q}_\ell)\) is Tate type for all \(i\) and all \(1\leq n\leq 9\) (see table 1). In this range the compactly supported Euler characteristics are given by: 
	\begin{align*}
		e_\mathrm{c}(\mathcal{X}_1,\mathbb{Q}_\ell) &= \mathbb{L}^2+\mathbb{L}\\
		e_\mathrm{c}(\mathcal{X}^{\times 2}_1,\mathbb{Q}_\ell) &= \mathbb{L}^3+3\mathbb{L}^2+\mathbb{L}-1\\
		e_\mathrm{c}(\mathcal{X}^{\times 3}_1,\mathbb{Q}_\ell) &= \mathbb{L}^4+6\mathbb{L}^3+6\mathbb{L}^2-2\mathbb{L}-3\\
		e_\mathrm{c}(\mathcal{X}^{\times 4}_1,\mathbb{Q}_\ell) &= \mathbb{L}^5+10\mathbb{L}^4+20\mathbb{L}^3+4\mathbb{L}^2-14\mathbb{L}-7\\
		e_\mathrm{c}(\mathcal{X}^{\times 5}_1,\mathbb{Q}_\ell) &= \mathbb{L}^6+15\mathbb{L}^5+50\mathbb{L}^4+40\mathbb{L}^3-30\mathbb{L}^2-49\mathbb{L}-15\\
		e_\mathrm{c}(\mathcal{X}^{\times 6}_1,\mathbb{Q}_\ell) &= \mathbb{L}^7+21\mathbb{L}^6+105\mathbb{L}^5+160\mathbb{L}^4-183\mathbb{L}^2-139\mathbb{L}-31\\
		e_\mathrm{c}(\mathcal{X}^{\times 7}_1,\mathbb{Q}_\ell) &= \mathbb{L}^8+28\mathbb{L}^7+196\mathbb{L}^6+469\mathbb{L}^5+280\mathbb{L}^4-427\mathbb{L}^3-700\mathbb{L}^2-356\mathbb{L}-63\\
		e_\mathrm{c}(\mathcal{X}^{\times 8}_1,\mathbb{Q}_\ell) &= \mathbb{L}^9+36\mathbb{L}^8+336\mathbb{L}^7+1148\mathbb{L}^6+1386\mathbb{L}^5-406\mathbb{L}^4-2436\mathbb{L}^3-2224\mathbb{L}^2-860\mathbb{L}-127\\
		e_\mathrm{c}(\mathcal{X}^{\times 9}_1,\mathbb{Q}_\ell) &= \mathbb{L}^{10}+45\mathbb{L}^9+540\mathbb{L}^8+2484\mathbb{L}^7+4662\mathbb{L}^6+1764\mathbb{L}^5-6090\mathbb{L}^4-9804\mathbb{L}^3-6372\mathbb{L}^2-2003\mathbb{L}-255
	\end{align*}
	The cohomology \(H^i(\mathcal{X}^{\times 10}_1,\mathbb{Q}_\ell)\) is Tate type for all \(i\neq 11\) (see table 1), whereas for \(i=11\) we have
	\begin{align*}
		H^{11}(\mathcal{X}^{\times 10}_1,\mathbb{Q}_\ell)=\mathbb{S}_{\Gamma(1)}[12]+\mathbb{L}^{11}+99\mathbb{L}^{10}+1925\mathbb{L}^9+12375\mathbb{L}^8+29700\mathbb{L}^7
	\end{align*}
	where \(\mathbb{S}_{\Gamma(1)}[12]\) is the \(2\)-dimensional Galois representation attached to the weight \(12\) cusp form \(\Delta\in S_{12}(\Gamma(1))\). In this case the compactly supported Euler characteristic is given by: 
	\begin{align*}
		e_\mathrm{c}(\mathcal{X}^{\times 10}_1,\mathbb{Q}_\ell) &= -\mathbb{S}_{\Gamma(1)}[12]\\
			&+ \mathbb{L}^{11}+55\mathbb{L}^{10}+825\mathbb{L}^9+4905\mathbb{L}^8+12870\mathbb{L}^7+12264\mathbb{L}^6\\
			&- 9240\mathbb{L}^5-33210\mathbb{L}^4-33495\mathbb{L}^3-17095\mathbb{L}^2-4553\mathbb{L}-511
	\end{align*}
	In particular the compactly supported Euler characteristic \(e_\mathrm{c}(\mathcal{X}^{\times n}_1,\mathbb{Q}_\ell)\) is not Tate type if \(n\geq 10\). 
	\end{theorem}
	\begin{proof}
	Follows by combining \ref{Leray} and \ref{Kunneth} with \ref{EichlerShimura1}. In this case the multiplicities \(m^{j,n}_\lambda\) are easily computed using the fact that
	\begin{align*}
		\mathbb{V}_{\lambda_1}\otimes\mathbb{V}_{\lambda_2}=\mathbb{V}_{\lambda_1+\lambda_2}\oplus\mathbb{V}_{\lambda_1+\lambda_2-2}\oplus\hdots\oplus\mathbb{V}_{|\lambda_1-\lambda_2|}
	\end{align*}
	To argue that \(e_\mathrm{c}(\mathcal{X}^{\times n}_1,\mathbb{Q}_\ell)\) is not Tate type if \(n\geq 10\) note that \(H^{11}(\mathcal{X}^{\times 10}_1,\mathbb{Q}_\ell)\) (which is not Tate type, owing to the irreducible \(2\)-dimensional contribution \(\mathbb{S}_{\Gamma(1)}[12]\) to \(H^1(\mathcal{A}_1,\mathbb{V}_{10})\)) appears as a summand in \(H^{11}(\mathcal{X}^{\times n}_1,\mathbb{Q}_\ell)\) for all \(n\geq 10\) by the K\"unneth formula. This contribution cannot be cancelled in the Euler characteristic: since the contribution occurs in \(H^i(\mathcal{X}^{\times n}_1,\mathbb{Q}_\ell)\) for \(i\) odd, any contribution leading to cancellation would have to occur in \(H^i(\mathcal{X}^{\times n}_1,\mathbb{Q}_\ell)\) for \(i\) even. Since \(H^*(\mathcal{A}_1,\mathbb{V}_\lambda)=0\) for \(\lambda>0\) odd, any contribution to \(H^i(\mathcal{X}^{\times n}_1,\mathbb{Q}_\ell)\) for \(i\) even would have to come from a contribution to \(H^0(\mathcal{A}_1,\mathbb{V}_\lambda)\) (since \(H^2(\mathcal{A}_1,\mathbb{V}_\lambda)=0\) for all \(\lambda\geq 0\)), but there are no irreducible \(2\)-dimensional contributions in this case: the only irreducible \(2\)-dimensional contributions come from the contribution \(\mathbb{S}_{\Gamma(1)}[\lambda+2]\) to \(H^1(\mathcal{A}_1,\mathbb{V}_\lambda)\). 
	\end{proof}
	
	We obtain the following corollary:
	
	\begin{corollary}
	\label{MGF1} The first \(9\) terms of the moment generating function \(M_{\#A_1(\mathbb{F}_q)}(t)\) are rational functions in \(q\):
	\begin{align*}
		1 &+ (\mathbf{q+1})t\\
		&+ (\mathbf{q^2+3q}+1-\tfrac{1}{q})\tfrac{t^2}{2!}\\
		&+ (\mathbf{q^3+6q^2}+6q-2-\tfrac{3}{q})\tfrac{t^3}{3!}\\
		&+ (\mathbf{q^4+10q^3}+20q^2+4q-14-\tfrac{7}{q})\tfrac{t^4}{4!}\\
		&+ (\mathbf{q^5+15q^4}+50q^3+40q^2-30q-49-\tfrac{15}{q})\tfrac{t^5}{5!}\\
		&+ (\mathbf{q^6+21q^5}+105q^4+160q^3-183q-139-\tfrac{31}{q})\tfrac{t^6}{6!}\\
		&+ (\mathbf{q^7+28q^6}+196q^5+469q^4+280q^3-427q^2-700q-356-\tfrac{63}{q})\tfrac{t^7}{7!}\\
		&+ (\mathbf{q^8+36q^7}+336q^6+1148q^5+1386q^4-406q^3-2436q^2-2224q-860-\tfrac{127}{q})\tfrac{t^8}{8!}\\
		&+ (\mathbf{q^9+45q^8}+540q^7+2484q^6+4662q^5+1764q^4-6090q^3-9804q^2-6372q-2003-\tfrac{255}{q})\tfrac{t^9}{9!}
	\end{align*}
	\end{corollary}
	
	Note that the first \(2\) coefficients in each of these terms (in bold) are consistent with \ref{Conjecture1}.
	
\section{Computations for $g=2$}

	Let \(\mathcal{A}_2\) be the moduli stack of principally polarized Abelian surfaces, which is a smooth Deligne-Mumford stack of dimension \(3\) over \(\mathbb{Z}\). Let \(\pi:\mathcal{X}_2\rightarrow\mathcal{A}_2\) be the universal Abelian surface over \(\mathcal{A}_2\) and let \(\mathbb{V}=\mathbb{R}^1\pi_*\mathbb{Q}_\ell\) be the \(\ell\)-adic local system on \(\mathcal{A}_2\) corresponding to the standard representation of \(\mathrm{Sp}_4\). For \(\lambda=(\lambda_1\geq\lambda_2\geq 0)\) a dominant integral highest weight for \(\mathrm{Sp}_4\) let \(\mathbb{V}_\lambda\) be the \(\ell\)-adic local system on \(\mathcal{A}_2\) corresponding to the irreducible representation of \(\mathrm{Sp}_4\) of highest weight \(\lambda\), occurring in \(\mathrm{Sym}^{\lambda_1-\lambda_2}(\mathbb{V})\otimes\mathrm{Sym}^{\lambda_2}(\wedge^2\mathbb{V})\). For \(\lambda_1+\lambda_2\) odd we have \(H^*(\mathcal{A}_2,\mathbb{V}_\lambda)=0\) since \(-\mathrm{id}\in\mathrm{Sp}_4(\mathbb{Z})\) acts by multiplication by \((-1)^{\lambda_1+\lambda_2}\) on the stalks of \(\mathbb{V}_{\lambda_1,\lambda_2}\). 
	
	Let \(\mathbb{S}_{\Gamma(1)}[\lambda_1-\lambda_2,\lambda_2+3]=\bigoplus_F\rho_F\) be the \(\ell\)-adic Galois representation corresponding to vector-valued Siegel cusp forms of weight \((\lambda_1-\lambda_2,\lambda_2+3)\) for \(\Gamma(1)=\mathrm{Sp}_4(\mathbb{Z})\): for each eigenform \(F\in S_{\lambda_1-\lambda_2,\lambda_2+3}(\Gamma(1))\) we have a \(4\)-dimensional \(\ell\)-adic Galois representation \(\rho_F\), and we have 
	\begin{align*}
		\mathrm{tr}(\mathrm{Frob}_p|\mathbb{S}_{\Gamma(1)}[\lambda_1-\lambda_2,\lambda_2+3])=\mathrm{tr}(T_p|S_{\lambda_1-\lambda_2,\lambda_2+3}(\Gamma(1)))
	\end{align*} 
	for every prime \(p\), which determines \(\mathbb{S}_{\Gamma(1)}[\lambda_1-\lambda_2,\lambda_2+3]\) as an element of the Grothendieck group of \(\ell\)-adic Galois representations. 
	
	As a representation of \(\mathrm{Gal}(\overline{\mathbb{Q}}/\mathbb{Q})\) the \(\ell\)-adic Galois representation \(\rho_F\) need not be irreducible: it is reducible for instance when \(F\in S_{0,k}(\Gamma(1))\) is the Saito-Kurokawa lift of a cusp form \(f\in S_{2k-2}(\Gamma(1))\) (see \cite[Theorem 21.1]{vdG} for a description of the Saito-Kurokawa lift), in which case \(\rho_F\simeq\rho_f+\mathbb{L}^{k-1}+\mathbb{L}^{k-2}\) up to semisimplification. On the other hand if \(F\in S_{\lambda_1-\lambda_2,\lambda_2+3}(\Gamma(1))\) is a vector-valued Siegel modular form of general type, the \(\ell\)-adic Galois representation \(\rho_F\) is irreducible as a representation of \(\mathrm{Gal}(\overline{\mathbb{Q}}/\mathbb{Q})\) and of \(\mathrm{Gal}(\overline{\mathbb{F}}_p/\mathbb{F}_p)\) (see \cite[Theorem I, Theorem II]{Weissauer}). Write \(\mathbb{S}^\mathrm{gen}_{\Gamma(1)}[\lambda_1-\lambda_2,\lambda_2+3]\) for the \(\ell\)-adic Galois representation corresponding to vector-valued Siegel cusp forms of general type. 
	
	By work of Petersen, using work of Harder \cite{Harder} and Flicker \cite{FlickerPGSp4} as input, we have the following:  

	\begin{proposition}
	\label{EichlerShimura2} \cite[Theorem 2.1]{Petersen} (compare to \cite[Conjecture 6.3]{BFvdG3}) for \(\lambda_1\geq\lambda_2\geq 0\) with \(\lambda_1+\lambda_2>0\) even we have
	\begin{align*}
		e_\mathrm{c}(\mathcal{A}_2,\mathbb{V}_{\lambda_1,\lambda_2}) &= -\mathbb{S}_{\Gamma(1)}[\lambda_1-\lambda_2,\lambda_2+3]+e_\mathrm{c,extr}(\mathcal{A}_2,\mathbb{V}_{\lambda_1,\lambda_2})
	\end{align*}
	as an element of the Grothendieck group of \(\ell\)-adic Galois representations, where \(e_\mathrm{c,extr}(\mathcal{A}_2,\mathbb{V}_{\lambda_1,\lambda_2})\) is given by
	\begin{align*}
		e_\mathrm{c,extr}(\mathcal{A}_2,\mathbb{V}_{\lambda_1,\lambda_2}) &= -s_{\Gamma(1)}[\lambda_1+\lambda_2+4]\mathbb{S}_{\Gamma(1)}[\lambda_1-\lambda_2+2]\mathbb{L}^{\lambda_2+1}\\
			&+ s_{\Gamma(1)}[\lambda_1-\lambda_2+2]-s_{\Gamma(1)}[\lambda_1+\lambda_2+4]\mathbb{L}^{\lambda_2+1}\\
			&+ \begin{cases} \mathbb{S}_{\Gamma(1)}[\lambda_2+2]+1 & \lambda_1\text{ even}\\ -\mathbb{S}_{\Gamma(1)}[\lambda_1+3] & \lambda_1\text{ odd} \end{cases}
	\end{align*}
	where \(s_{\Gamma(1)}[k]\) is the dimension of the space of cusp forms of weight \(k\) for \(\Gamma(1)=\mathrm{SL}_2(\mathbb{Z})\) (where we set \(\mathbb{S}_{\Gamma(1)}[2]:=-\mathbb{L}-1\) and \(s_{\Gamma(1)}[2]:=-1\)).
	\end{proposition}
	
	This remains true for \((\lambda_1,\lambda_2)=(0,0)\) if we set \(\mathbb{S}_{\Gamma(1)}[0,3]:=-\mathbb{L}^3-\mathbb{L}^2-\mathbb{L}-1\): by \cite[Corollary 5.2.3]{LeeWeintraub} we have 
	\begin{align*}
		e_\mathrm{c}(\mathcal{A}_2,\mathbb{Q}_\ell)=\mathbb{L}^3+\mathbb{L}^2
	\end{align*}
	We will use the following values for the Euler characteristics \(e_\mathrm{c}(\mathcal{A}_2,\mathbb{V}_{\lambda_1,\lambda_2})\), which are obtained by combining \ref{EichlerShimura2} with the vanishing of the spaces \(S_{\lambda_1-\lambda_2,\lambda_2+3}(\Gamma(1))\) for all \(\lambda_1\geq\lambda_2\geq 0\) with \(\lambda_1,\lambda_2\leq 7\) except for \(\lambda_1=\lambda_2=7\):
	\begin{center}\begin{tabular}{|c|c|}
		\hline
		\((\lambda_1,\lambda_2)\) & \(e_\mathrm{c}(\mathcal{A}_2,\mathbb{V}_{\lambda_1,\lambda_2})\)\\
		\hline\hline
		\((0,0)\) & \(\mathbb{L}^3+\mathbb{L}^2\)\\
		\hline
		\((2,0)\) & \(-\mathbb{L}\)\\
		\((1,1)\) & \(-1\)\\
		\hline
		\((4,0)\) & \(-\mathbb{L}\)\\
		\((3,1)\) & \(0\)\\
		\((2,2)\) & \(0\)\\
		\hline
		\((6,0)\) & \(-\mathbb{L}\)\\
		\((5,1)\) & \(0\)\\
		\((4,2)\) & \(1\)\\
		\((3,3)\) & \(-1\)\\
		\hline
	\end{tabular}\qquad\begin{tabular}{|c|c|}
		\hline
		\((\lambda_1,\lambda_2)\) & \(e_\mathrm{c}(\mathcal{A}_2,\mathbb{V}_{\lambda_1,\lambda_2})\)\\
		\hline\hline
		\((7,1)\) & \(-\mathbb{L}^2\)\\
		\((6,2)\) & \(-\mathbb{L}^3+1\)\\
		\((5,3)\) & \(-\mathbb{L}^4\)\\
		\((4,4)\) & \(\mathbb{L}^6\)\\
		\hline
		\((7,3)\) & \(0\)\\
		\((6,4)\) & \(1\)\\
		\((5,5)\) & \(-1\)\\
		\hline
		\((7,5)\) & \(-\mathbb{L}^6\)\\
		\((6,6)\) & \(\mathbb{L}^8\)\\
		\hline
		\((7,7)\) & \(-\mathbb{S}_{\Gamma(1)}[18]-\mathbb{L}^8-1\)\\
		\hline
	\end{tabular}\end{center}
	The space \(S_{0,10}(\Gamma(1))\) is spanned by the Igusa cusp form (see \cite{OberdieckPixton}):
	\begin{align*}
		\chi_{10} &= (q^{-1}-2+q)q_1q_2-(2q^{-2}+16q^{-1}-36+16q+2q^2)(q^2_1q_2+q_1q^2_2)\\
			&+ (q^{-3}+36q^{-2}+99q^{-1}-272+99q+36q^2+q^3)(q^3_1q_2+q_1q^3_2)\\
			&+ (4q^{-3}+72q^{-2}+252q^{-1}-656+252q+72q^2+4q^3)q^2_1q^2_2+\hdots
	\end{align*}
	which is a Saito-Kurokawa lift of the weight \(18\) cusp form \(f_{18}=\Delta E_6\in S_{18}(\Gamma(1))\) and contributes an irreducible \(2\)-dimensional \(\ell\)-adic Galois representation \(\mathbb{S}_{\Gamma(1)}[18]\) to \(H^3(\mathcal{A}_2,\mathbb{V}_{7,7})\) (see for example \cite[4.3.5]{Petersen}) with the property that \(\mathrm{tr}(\mathrm{Frob}_p|\mathbb{S}_{\Gamma(1)}[18])=\lambda_p(f_{18})\) (the eigenvalue of the Hecke operator \(T_p\) on \(f_{18}\)), which is not polynomial in \(p\); the remaining summands \(\mathbb{L}^9\) and \(\mathbb{L}^8\) of the \(4\)-dimensional \(\ell\)-adic Galois representation \(\mathbb{S}_{\Gamma(1)}[0,10]=\mathbb{S}_{\Gamma(1)}[18]+\mathbb{L}^9+\mathbb{L}^8\) do not contribute to \(H^3(\mathcal{A}_2,\mathbb{V}_{7,7})\). 
	
	We will use another contribution which does not appear in the above table but which was mentioned in the introduction. The space \(S_{6,8}(\Gamma(1))\) is spanned by the vector-valued cusp form (see \cite[Section 8]{CleryvdG})
	\begin{align*}
		\chi_{6,8} &= \left(\begin{smallmatrix} 0\\ 0\\ q^{-1}-2+q\\ 2(q-q^{-1})\\ q^{-1}-2+q\\ 0\\ 0 \end{smallmatrix}\right)q_1q_2+\left(\begin{smallmatrix} 0\\ 0\\ -2(q^{-2}+8q^{-1}-18+8q+q^2)\\ 8(q^{-2}+4q^{-1}-4q-q^2)\\ -2(7q^{-2}-4q^{-1}-6-4q+7q^2)\\ 12(q^{-2}-2q^{-1}+2q-q^2)\\ -4(q^{-2}-4q^{-1}+6-4q+q^2) \end{smallmatrix}\right)q_1q^2_2\\
			&+ \left(\begin{smallmatrix} -4(q^{-2}-4q^{-1}+6-4q+q^2)\\ 12(q^{-2}-2q^{-1}+2q-q^2)\\ -2(7q^{-2}-4q^{-1}-6-4q+7q^2)\\ 8(q^{-2}+4q^{-1}-4q-q^2)\\ -2(q^{-2}+8q^{-1}-18+8q+q^2)\\ 0\\ 0 \end{smallmatrix}\right)q^2_1q_2+\left(\begin{smallmatrix} 16(q^{-3}-9q^{-1}+16-9q+q^3)\\ -72(q^{-3}-3q^{-1}+3q-q^3)\\ 128(q^{-3}-2+q^3)\\ -144(q^{-3}+5q^{-1}-5q-q^3)\\ 128(q^{-3}-2+q^3)\\ -72(q^{-3}-3q^{-1}+3q-q^3)\\ 16(q^{-3}-9q^{-1}+16-9q+q^3) \end{smallmatrix}\right)q^2_1q^2_2+\hdots
	\end{align*}
	which is of general type and contributes an irreducible \(4\)-dimensional \(\ell\)-adic Galois representation \(\mathbb{S}_{\Gamma(1)}[6,8]\) to \(H^3_\mathrm{c}(\mathcal{A}_2,\mathbb{V}_{11,5})\) (see for example \cite[4.3.1]{Petersen}) with the property that \(\mathrm{tr}(\mathrm{Frob}_p|\mathbb{S}_{\Gamma(1)}[6,8])=\lambda_p(\chi_{6,8})\) (the eigenvalue of the Hecke operator \(T_p\) acting on \(\chi_{6,8}\)) which is not polynomial in \(p\). 
	

	We obtain the following result: 
	
	\pagebreak
	
	\begin{theorem}
	\label{Theorem2} The cohomology \(H^i(\mathcal{X}^{\times n}_2,\mathbb{Q}_\ell)\) is Tate type for all \(i\) and all \(1\leq n\leq 6\) (see table 2). In this range the compactly supported Euler characteristics are given by: 
	\begin{align*}
		e_\mathrm{c}(\mathcal{X}_2,\mathbb{Q}_\ell) &= \mathbb{L}^5+2\mathbb{L}^4+2\mathbb{L}^3+\mathbb{L}^2-1\\
		e_\mathrm{c}(\mathcal{X}^{\times 2}_2,\mathbb{Q}_\ell) &= \mathbb{L}^7+4\mathbb{L}^6+9\mathbb{L}^5+9\mathbb{L}^4+3\mathbb{L}^3-5\mathbb{L}^2-5\mathbb{L}-3\\
		e_\mathrm{c}(\mathcal{X}^{\times 3}_2,\mathbb{Q}_\ell) &= \mathbb{L}^9+7\mathbb{L}^8+27\mathbb{L}^7+49\mathbb{L}^6+46\mathbb{L}^5+3\mathbb{L}^4-42\mathbb{L}^3-53\mathbb{L}^2-24\mathbb{L}-7\\
		e_\mathrm{c}(\mathcal{X}^{\times 4}_2,\mathbb{Q}_\ell) &= \mathbb{L}^{11}+11\mathbb{L}^{10}+65\mathbb{L}^9+191\mathbb{L}^8+320\mathbb{L}^7+257\mathbb{L}^6\\
			&- 65\mathbb{L}^5-425\mathbb{L}^4-474\mathbb{L}^3-273\mathbb{L}^2-73\mathbb{L}-14\\
		e_\mathrm{c}(\mathcal{X}^{\times 5}_2,\mathbb{Q}_\ell) &= \mathbb{L}^{13}+16\mathbb{L}^{12}+135\mathbb{L}^{11}+590\mathbb{L}^{10}+1525\mathbb{L}^9+2292\mathbb{L}^8+1527\mathbb{L}^7\\
			&- 1285\mathbb{L}^6-4219\mathbb{L}^5-4730\mathbb{L}^4-2814\mathbb{L}^3-923\mathbb{L}^2-135\mathbb{L}-21\\
		e_\mathrm{c}(\mathcal{X}^{\times 6}_2,\mathbb{Q}_\ell) &= \mathbb{L}^{15}+22\mathbb{L}^{14}+252\mathbb{L}^{13}+1540\mathbb{L}^{12}+5683\mathbb{L}^{11}+13035\mathbb{L}^{10}+17779\mathbb{L}^9+8660\mathbb{L}^8\\
			&- 17614\mathbb{L}^7-44408\mathbb{L}^6-48770\mathbb{L}^5-30667\mathbb{L}^4-10437\mathbb{L}^3-1391\mathbb{L}^2+142\mathbb{L}+2
	\end{align*}
	The cohomology \(H^i(\mathcal{X}^{\times 7}_2,\mathbb{Q}_\ell)\) is Tate type for all \(i\neq 17\) (see table 2), whereas for \(i=17\) we have
	\begin{align*}
		H^{17}(\mathcal{X}^{\times 7}_2,\mathbb{Q}_\ell)=\mathbb{S}_{\Gamma(1)}[18]+\mathbb{L}^{17}+1176\mathbb{L}^{15}+63700\mathbb{L}^{13}+6860\mathbb{L}^{12}+321048\mathbb{L}^{11}+294440\mathbb{L}^{10}+\mathbb{L}^9
	\end{align*}
	where \(\mathbb{S}_{\Gamma(1)}[18]\) is the \(2\)-dimensional \(\ell\)-adic Galois representation attached to the weight \(18\) cusp form \(f_{18}=\Delta E_6\in S_{18}(\Gamma(1))\). In this case the compactly supported Euler characteristic is given by: 
	\begin{align*}
		e_\mathrm{c}(\mathcal{X}^{\times 7}_2,\mathbb{Q}_\ell) &= -\mathbb{S}_{\Gamma(1)}[18]\\
			&+ \mathbb{L}^{17}+29\mathbb{L}^{16}+434\mathbb{L}^{15}+3542\mathbb{L}^{14}+17717\mathbb{L}^{13}+56924\mathbb{L}^{12}+118692\mathbb{L}^{11}+145567\mathbb{L}^{10}+37850\mathbb{L}^9\\
			&- 226570\mathbb{L}^8-487150\mathbb{L}^7-529851\mathbb{L}^6-342930\mathbb{L}^5-121324\mathbb{L}^4-9491\mathbb{L}^3+9018\mathbb{L}^2+3164\mathbb{L}+223
	\end{align*}
	In particular the compactly supported Euler characteristic \(e_\mathrm{c}(\mathcal{X}^{\times n}_2,\mathbb{Q}_\ell)\) is not Tate type if \(n\geq 7\).
	\end{theorem}
	\begin{proof}
	Follows by combining \ref{Leray} and \ref{Kunneth} with \ref{EichlerShimura2}. In this case we computed the multiplicities \(m^{j,n}_\lambda\) with a SAGE program (available on request). 
	
	To argue that \(e_\mathrm{c}(\mathcal{X}^{\times n}_2,\mathbb{Q}_\ell)\) is not Tate type if \(n\geq 7\) note that \(H^{17}(\mathcal{X}^{\times 7}_2,\mathbb{Q}_\ell)\) (which is not Tate type, owing to the irreducible \(2\)-dimensional contribution \(\mathbb{S}_{\Gamma(1)}[18]\) to \(H^3(\mathcal{A}_2,\mathbb{V}_{7,7})\)) appears as a summand in \(H^{17}(\mathcal{X}^{\times n}_2,\mathbb{Q}_\ell)\) for all \(n\geq 7\) by the K\"unneth formula. This contribution cannot be cancelled in the Euler characteristic, at least for \(7\leq n\leq 15\): since the contribution occurs in \(H^i(\mathcal{X}^{\times n}_2,\mathbb{Q}_\ell)\) for \(i\) odd, any contribution leading to cancellation would have to occur in \(H^i(\mathcal{X}^{\times n}_2,\mathbb{Q}_\ell)\) for \(i\) even. Since \(H^*(\mathcal{A}_2,\mathbb{V}_{\lambda_1,\lambda_2})=0\) for \(\lambda_1+\lambda_2>0\) odd, any contribution to \(H^i(\mathcal{X}^{\times n}_2,\mathbb{Q}_\ell)\) for \(i\) even would have to come from a contribution to \(H^j(\mathcal{A}_2,\mathbb{V}_{\lambda_1,\lambda_2})\) for \(j=0,2,4\) (since \(H^6(\mathcal{A}_2,\mathbb{V}_{\lambda_1,\lambda_2})=0\) for all \(\lambda_1\geq\lambda_2\geq 0\)). The only irreducible \(2\)-dimensional contributions that occur in this way come from the contribution \(\mathbb{S}_{\Gamma(1)}[\lambda_2+2]\mathbb{L}^{\lambda_1+2}\) to \(H^4(\mathcal{A}_2,\mathbb{V}_{\lambda_1,\lambda_2})\) (Poincare dual to the contribution \(\mathbb{S}_{\Gamma(1)}[\lambda_2+2]\) to \(H^2_\mathrm{c}(\mathcal{A}_2,\mathbb{V}_{\lambda_1,\lambda_2})\) in \cite[Theorem 2.1]{Petersen}), which would require \(\lambda_2=16\) for cancellation. 
	
	Now note that \(H^{19}(\mathcal{X}^{\times 11}_2,\mathbb{Q}_\ell)\) (which is not Tate type, owing to the irreducible \(4\)-dimensional contribution \(\mathbb{S}_{\Gamma(1)}[6,8]\) to \(H^3(\mathcal{A}_2,\mathbb{V}_{11,5})\)) appears as a summand in \(H^{19}(\mathcal{X}^{\times n}_2,\mathbb{Q}_\ell)\) for all \(n\geq 11\) by the K\"unneth formula. This contribution cannot be cancelled in the Euler characteristic: by the same reasoning as above any contribution leading to cancellation would have to come from a contribution to \(H^j(\mathcal{A}_2,\mathbb{V}_{\lambda_1,\lambda_2})\) for \(j=0,2,4\), but there are no irreducible \(4\)-dimensional contributions in this case: the only irreducible \(4\)-dimensional contributions come from the contribution \(\mathbb{S}^\mathrm{gen}_{\Gamma(1)}[\lambda_1-\lambda_2,\lambda_2+3]\) to \(H^3(\mathcal{A}_2,\mathbb{V}_{\lambda_1,\lambda_2})\) (Poincare dual to the contribution \(\mathbb{S}^\mathrm{gen}_{\Gamma(1)}[\lambda_1-\lambda_2,\lambda_2+3]\) to \(H^3_\mathrm{c}(\mathcal{A}_2,\mathbb{V}_{\lambda_1,\lambda_2})\) in \cite[Theorem 2.1]{Petersen}). 
	\end{proof}
	
	Note that the contribution \(\mathbb{S}_{\Gamma(1)}[18]\) should always persist, but we cannot argue this without estimates on the multiplicities \(m^{j,n}_\lambda\). 
	
	We obtain the following corollary:
	
	\begin{corollary}
	\label{MGF2} The first \(6\) terms of the moment generating function \(M_{\#A_2(\mathbb{F}_q)}(t)\) are rational functions in \(q\): 
	\begin{align*}
		1 &+ (\mathbf{q^2+q+1}-\frac{1}{q^3+q^2})t\\
			&+ (\mathbf{q^4+3q^3+6q^2}+3q-\frac{5q^2+5q+3}{q^3+q^2})\frac{t^2}{2!}\\
			&+ (\begin{matrix} \mathbf{q^6+6q^5+21q^4}+28q^3\\ +18q^2-15q-27 \end{matrix}-\frac{26q^2+24q+7}{q^3+q^2})\frac{t^3}{3!}\\
			&+ (\begin{matrix} \mathbf{q^8+10q^7+55q^6}+136q^5+184q^4\\ +73q^3-138q^2-287q-187 \end{matrix}-\frac{86q^2+73q+14}{q^3+q^2})\frac{t^4}{4!}\\
			&+ (\begin{matrix} \mathbf{q^{10}+15q^9+120q^8}+470q^7+1055q^6+1237q^5\\ +290q^4-1575q^3-2644q^2-2086q-728 \end{matrix}-\frac{195q^2+135q+21}{q^3+q^2})\frac{t^5}{5!}\\
			&+ (\begin{matrix} \mathbf{q^{12}+21q^{11}+231q^{10}}+1309q^9+4374q^8+8661q^7+9118q^6\\ -458 q^5-17156q^4-27252q^3-21518q^2-9149q-1288 \end{matrix}-\frac{103q^2-142q-2}{q^3+q^2})\frac{t^6}{6!}
	\end{align*}
	\end{corollary}
	
	Note that the first \(3\) coefficients in each of these terms (in bold) are consistent with \ref{Conjecture1}. 
	
\section{Computations for $g=3$}

	Let \(\mathcal{A}_3\) be the moduli stack of principally polarized Abelian threefolds, which is a smooth Deligne-Mumford stack of dimension \(6\) over \(\mathbb{Z}\). Let \(\pi:\mathcal{X}_3\rightarrow\mathcal{A}_3\) be the universal Abelian threefold over \(\mathcal{A}_3\) and let \(\mathbb{V}=\mathbb{R}^1\pi_*\mathbb{Q}_\ell\) be the \(\ell\)-adic local system on \(\mathcal{A}_3\) corresponding to the standard representation of \(\mathrm{Sp}_6\). For \(\lambda=(\lambda_1\geq\lambda_2\geq\lambda_3\geq 0)\) a dominant integral highest weight for \(\mathrm{Sp}_6\) let \(\mathbb{V}_\lambda\) be the \(\ell\)-adic local system on \(\mathcal{A}_3\) corresponding to the irreducible representation of \(\mathrm{Sp}_6\) of highest weight \(\lambda\), occurring in \(\mathrm{Sym}^{\lambda_1-\lambda_2}(\mathbb{V})\otimes\mathrm{Sym}^{\lambda_2-\lambda_3}(\wedge^2\mathbb{V})\otimes\mathrm{Sym}^{\lambda_3}(\wedge^3\mathbb{V})\). For \(\lambda_1+\lambda_2+\lambda_3\) odd we have \(H^*(\mathcal{A}_3,\mathbb{V}_\lambda)=0\) since \(-\mathrm{id}\in\mathrm{Sp}_6(\mathbb{Z})\) acts by multiplication by \((-1)^{\lambda_1+\lambda_2+\lambda_3}\) on the stalks of \(\mathbb{V}_{\lambda_1,\lambda_2,\lambda_3}\). 
	
	Let \(\mathbb{S}_{\Gamma(1)}[\lambda_1-\lambda_2,\lambda_2-\lambda_3,\lambda_3+4]=\bigoplus_F\rho_F\) be the \(\ell\)-adic Galois representation corresponding to vector-valued Siegel cusp forms of weight \((\lambda_1-\lambda_2,\lambda_2-\lambda_3,\lambda_3+4)\) for \(\Gamma(1)=\mathrm{Sp}_6(\mathbb{Z})\): for each eigenform \(F\in S_{\lambda_1-\lambda_2,\lambda_2-\lambda_3,\lambda_3+4}(\Gamma(1))\) we have an \(8\)-dimensional \(\ell\)-adic Galois representation \(\rho_F\), and we have 
	\begin{align*}
		\mathrm{tr}(\mathrm{Frob}_p|\mathbb{S}_{\Gamma(1)}[\lambda_1-\lambda_2,\lambda_2-\lambda_3,\lambda_3+4])=\mathrm{tr}(T_p|S_{\lambda_1-\lambda_2,\lambda_2-\lambda_3,\lambda_3+4}(\Gamma(1)))
	\end{align*} 
	for every prime \(p\), which determines \(\mathbb{S}_{\Gamma(1)}[\lambda_1-\lambda_2,\lambda_2-\lambda_3,\lambda_3+4]\) as an element of the Grothendieck group of \(\ell\)-adic Galois representations. 
	
	As a representation of \(\mathrm{Gal}(\overline{\mathbb{Q}}/\mathbb{Q})\) the \(\ell\)-adic Galois representation \(\rho_F\) need not be irreducible, for example if \(F\) is one of the lifts predicted by \cite[Conjecture 7.7]{BFvdG3}. On the other hand if \(F\in S_{\lambda_1-\lambda_2,\lambda_2-\lambda_3,\lambda_3+4}(\Gamma(1))\) is a vector-valued Siegel cusp form of general type, the \(\ell\)-adic Galois representation \(\rho_F\) is predicted to be irreducible as a representation of \(\mathrm{Gal}(\overline{\mathbb{Q}}/\mathbb{Q})\) and of \(\mathrm{Gal}(\overline{\mathbb{F}}_p/\mathbb{F}_p)\). Write \(\mathbb{S}^\mathrm{gen}_{\Gamma(1)}[\lambda_1-\lambda_2,\lambda_2-\lambda_3,\lambda_3+4]\) for the \(\ell\)-adic Galois representation corresponding to vector-valued Siegel cusp forms of general type. 
	
	By work of Bergstr\"om-Faber-van der Geer, one conjectures the following: 
	
	\begin{conjecture}
	\label{EichlerShimura3} \cite[Conjecture 7.1]{BFvdG3} For \(\lambda_1\geq\lambda_2\geq\lambda_3\) with \(\lambda_1+\lambda_2+\lambda_3>0\) even we have
	\begin{align*}
		e_\mathrm{c}(\mathcal{A}_3,\mathbb{V}_{\lambda_1,\lambda_2,\lambda_3}) &= \mathbb{S}_{\Gamma(1)}[\lambda_1-\lambda_2,\lambda_2-\lambda_3,\lambda_3+4]+e_\mathrm{c,extr}(\mathcal{A}_3,\mathbb{V}_{\lambda_1,\lambda_2,\lambda_3})
	\end{align*}
	as an element of the Grothendieck group of \(\ell\)-adic Galois representations where \(e_\mathrm{c,extr}(\mathcal{A}_3,\mathbb{V}_{\lambda_1,\lambda_2,\lambda_3}) \) is given by 
	\begin{align*}
		e_\mathrm{c,extr}(\mathcal{A}_3,\mathbb{V}_{\lambda_1,\lambda_2,\lambda_3}) &= -e_\mathrm{c}(\mathcal{A}_2,\mathbb{V}_{\lambda_1+1,\lambda_2+1})-e_\mathrm{c,extr}(\mathcal{A}_2,\mathbb{V}_{\lambda_1+1,\lambda_2+1})\otimes\mathbb{S}_{\Gamma(1)}[\lambda_3+2]\\
			&+ e_\mathrm{c}(\mathcal{A}_2,\mathbb{V}_{\lambda_1+1,\lambda_3})+e_\mathrm{c,extr}(\mathcal{A}_2,\mathbb{V}_{\lambda_1+1,\lambda_3})\otimes\mathbb{S}_{\Gamma(1)}[\lambda_2+3]\\
			&- e_\mathrm{c}(\mathcal{A}_2,\mathbb{V}_{\lambda_2,\lambda_3})-e_\mathrm{c,extr}(\mathcal{A}_2,\mathbb{V}_{\lambda_2,\lambda_3})\otimes\mathbb{S}_{\Gamma(1)}[\lambda_1+4]
	\end{align*}
	\end{conjecture}
	
	This remains true for \((\lambda_1,\lambda_2,\lambda_3)=(0,0,0)\) if we set \(\mathbb{S}_{\Gamma(1)}[0,0,4]:=\mathbb{L}^6+\mathbb{L}^5+\mathbb{L}^4+2\mathbb{L}^3+\mathbb{L}^2+\mathbb{L}+1\): by \cite[Theorem 1]{Hain} we have 
	\begin{align*}
		e_\mathrm{c}(\mathcal{A}_3,\mathbb{Q}_\ell)=\mathbb{L}^6+\mathbb{L}^5+\mathbb{L}^4+\mathbb{L}^3+1
	\end{align*}
	As explained in \cite[Section 8]{BFvdG3} this conjecture was made after extensive point counts for curves up to genus \(3\) over finite fields. In particular by \cite[Remark 8.2]{BFvdG3} the conjecture is true for all \((\lambda_1,\lambda_2,\lambda_3)\) with \(\lambda_1+\lambda_2+\lambda_3\leq 6\) on the basis of these point counts since \(S_{\lambda_1-\lambda_2,\lambda_2-\lambda_3,\lambda_3+4}(\Gamma(1))\) has dimension \(0\) in these cases by \cite{Taibi}. In view of \cite[Theorem 1.9]{CanningLarsonPayne2}, using the classification results of Chevevier-Ta\"ibi \cite{ChenevierTaibi}, the conjecture is true for all \((\lambda_1,\lambda_2,\lambda_3)\) with \(\lambda_1+\lambda_2+\lambda_3\leq 10\) on the basis of these point counts. The conjecture is claimed to be proven unconditionally by unpublished work of Ta\"ibi \cite{TaibiAg}. 
	
	We will use the following values for the Euler characteristics \(e_\mathrm{c}(\mathcal{A}_3,\mathbb{V}_{\lambda_1,\lambda_2,\lambda_3})\), which are obtained by combining \ref{EichlerShimura3} with the vanishing \(S_{\lambda_1-\lambda_2,\lambda_2-\lambda_3,\lambda_3+4}(\Gamma(1))\) for all \(\lambda_1\geq\lambda_2\geq\lambda_3\geq 0\) with \(\lambda_1,\lambda_2,\lambda_3\leq 6\) obtained by \cite{Taibi} (compare to the tables at the end of \cite{BFvdG3}): 
	\begin{center}\begin{tabular}{|c|c|}
		\hline
		\((\lambda_1,\lambda_2,\lambda_3)\) & \(e_\mathrm{c}(\mathcal{A}_3,\mathbb{V}_{\lambda_1,\lambda_2,\lambda_3})\)\\
		\hline\hline
		\((0,0,0)\) & \(\mathbb{L}^6+\mathbb{L}^5+\mathbb{L}^4+\mathbb{L}^3+1\)\\
		\hline
		\((2,0,0)\) & \(-\mathbb{L}^3-\mathbb{L}^2\)\\
		\((1,1,0)\) & \(-\mathbb{L}\)\\
		\hline
		\((4,0,0)\) & \(-\mathbb{L}^3-\mathbb{L}^2\)\\
		\((3,1,0)\) & \(0\)\\
		\((2,2,0)\) & \(0\)\\
		\((2,1,1)\) & \(1\)\\
		\hline
		\((6,0,0)\) & \(-2\mathbb{L}^3-\mathbb{L}^2\)\\
		\((5,1,0)\) & \(-\mathbb{L}^4\)\\
		\((4,2,0)\) & \(-\mathbb{L}^5+\mathbb{L}\)\\
		\((4,1,1)\) & \(1\)\\
		\((3,3,0)\) & \(\mathbb{L}^7-\mathbb{L}\)\\
		\((3,2,1)\) & \(0\)\\
		\((2,2,2)\) & \(1\)\\
		\hline
		\((6,2,0)\) & \(\mathbb{L}\)\\
		\((6,1,1)\) & \(-\mathbb{L}^2+1\)\\
		\((5,3,0)\) & \(0\)\\
		\((5,2,1)\) & \(0\)\\
		\((4,4,0)\) & \(0\)\\
		\((4,3,1)\) & \(0\)\\
		\((4,2,2)\) & \(\mathbb{L}^4\)\\
		\((3,3,2)\) & \(-\mathbb{L}^6+1\)\\
		\hline
	\end{tabular}\qquad\begin{tabular}{|c|c|}
		\hline
		\((\lambda_1,\lambda_2,\lambda_3)\) & \(e_\mathrm{c}(\mathcal{A}_3,\mathbb{V}_{\lambda_1,\lambda_2,\lambda_3})\)\\
		\hline\hline
		\((6,4,0)\) & \(-\mathbb{L}^7+\mathbb{L}\)\\
		\((6,3,1)\) & \(-\mathbb{L}^2\)\\
		\((6,2,2)\) & \(0\)\\
		\((5,5,0)\) & \(\mathbb{L}^9-\mathbb{L}\)\\
		\((5,4,1)\) & \(0\)\\
		\((5,3,2)\) & \(-\mathbb{L}^3\)\\
		\((4,4,2)\) & \(0\)\\
		\((4,3,3)\) & \(-\mathbb{L}^4+1\)\\
		\hline
		\((6,6,0)\) & \(\mathbb{S}_{\Gamma(1)}[0,10]+\mathbb{L}^{10}\)\\
		\((6,5,1)\) & \(-\mathbb{L}^2\)\\
		\((6,4,2)\) & \(\mathbb{L}^6-1\)\\
		\((6,3,3)\) & \(1\)\\
		\((5,5,2)\) & \(-\mathbb{L}^8-\mathbb{L}^3+1\)\\
		\((5,4,3)\) & \(0\)\\
		\((4,4,4)\) & \(-\mathbb{L}^6+1\)\\
		\hline
		\((6,6,2)\) & \(\mathbb{S}_{\Gamma(1)}[0,10]-\mathbb{L}^9+\mathbb{L}^3\)\\
		\((6,5,3)\) & \(\mathbb{L}^4\)\\
		\((6,4,4)\) & \(0\)\\
		\((5,5,4)\) & \(-\mathbb{L}^8+1\)\\
		\hline
		\((6,6,4)\) & \(\mathbb{S}_{\Gamma(1)}[0,10]-\mathbb{L}^9\)\\
		\((6,5,5)\) & \(-\mathbb{L}^6+1\)\\
		\hline
		\((6,6,6)\) & \(\mathbb{S}_{\Gamma(1)}[0,10]-\mathbb{L}^9-\mathbb{L}^8+1\)\\
		\hline
	\end{tabular}\end{center}
	We will use another contribution which does not appear n the above table. For \(\lambda=(9,6,3)\) we have a contribution from an \(8\)-dimensional \(\ell\)-adic Galois representation \(\mathbb{S}_{\Gamma(1)}[3,3,7]\) which decomposes into a \(1\)-dimensional \(\ell\)-adic Galois representation of Tate type and an irreducible \(7\)-dimensional \(\ell\)-adic Galois representation (see \cite[Example 9.1]{BFvdG3}), which is explained by a functorial lift from the exceptional group \(\mathrm{G}_2\) predicted by \cite{GrossSavin}. 
	
	The Langlands correspondence predicts in this case that an irreducible \(8\)-dimensional Galois representation \(\rho:\mathrm{Gal}(\overline{\mathbb{Q}}/\mathbb{Q})\rightarrow\mathrm{GL}_8(\overline{\mathbb{Q}}_\ell)\) (which is the composition of a \(\mathrm{Spin}_7\) Galois representation \(\rho':\mathrm{Gal}(\overline{\mathbb{Q}}/\mathbb{Q})\rightarrow\mathrm{Spin}_7(\overline{\mathbb{Q}}_\ell)=\widehat{\mathrm{PGSp}}_6\) with the \(8\)-dimensional spin representation \(\mathrm{spin}:\mathrm{Spin}_7(\overline{\mathbb{Q}}_\ell)\rightarrow\mathrm{GL}_8(\overline{\mathbb{Q}}_\ell)\)) contributing to the cohomology \(H^*(\mathcal{A}_3,\mathbb{V}_\lambda)\) must come from a packet of cuspidal automorphic representations \(\pi\) of \(\mathrm{PGSp}_6(\mathbb{A}_\mathbb{Q})\) with \(\pi_\infty|_{\mathrm{Sp}_6(\mathbb{R})}\) varying over all members of a discrete series L-packet. As the \((\mathfrak{sp}_6,\mathrm{U}(3))\)-cohomology of such discrete series representations is concentrated in degree \(3\) by \cite{VoganZuckerman}, such a contribution can only occur in \(H^6(\mathcal{A}_3,\mathbb{V}_\lambda)\). 
	
	As explained in \cite{GrossSavin}, any such \(\rho':\mathrm{Gal}(\overline{\mathbb{Q}}/\mathbb{Q})\rightarrow\mathrm{Spin}_7(\overline{\mathbb{Q}}_\ell)\) factoring through the inclusion \(\widehat{\mathrm{G}}_2=\mathrm{G}_2(\overline{\mathbb{Q}}_\ell)\hookrightarrow\mathrm{Spin}_7(\overline{\mathbb{Q}}_\ell)=\widehat{\mathrm{PGSp}}_6\) of the stabilizer of a non-isotropic vector in the \(8\)-dimensional spin representation must come from a packet of cuspidal automorphic representations \(\pi\) of \(\mathrm{G}_2(\mathbb{A}_\mathbb{Q})\) which lifts to a packet of cuspidal automorphic representations \(\pi'\) of \(\mathrm{PGSp}_6(\mathbb{A}_\mathbb{Q})\) with \(\pi'_\infty|_{\mathrm{Sp}_6(\mathbb{R})}\) varying over all but one member of a discrete series L-packet, and again such a contribution can only occur in \(H^6(\mathcal{A}_3,\mathbb{V}_\lambda)\); the remaining \(1\)-dimensional Tate-type contribution comes from the cycle class of a Hilbert modular threefold in this Siegel modular \(6\)-fold. 
	
	We record these predictions as the following conjecture:
	
	\begin{conjecture}
	\label{MiddleGalois} Any irreducible \(\ell\)-adic Galois representation of dimension \(7\) or \(8\) occurring in \(H^*(\mathcal{A}_3,\mathbb{V}_\lambda)\) can only occur in \(H^6(\mathcal{A}_3,\mathbb{V}_\lambda)\). 
	\end{conjecture}
	
	We obtain the following result, which is unconditional for \(1\leq n\leq 3\) on the basis of point counts (but is very much conditional on the above conjectures in the case \(n\geq 4\)):
	
	\begin{theorem}
	\label{Theorem3} Assume conjectures \ref{EichlerShimura3} and \ref{MiddleGalois}. Then the compactly supported Euler characteristic \(e_\mathrm{c}(\mathcal{X}^{\times n}_3,\mathbb{Q}_\ell)\) is Tate type for all \(1\leq n\leq 5\), and are given by:
	\begin{align*}
		e_\mathrm{c}(\mathcal{X}_3,\mathbb{Q}_\ell) &= \mathbb{L}^9+2\mathbb{L}^8+3\mathbb{L}^7+4\mathbb{L}^6+3\mathbb{L}^5+2\mathbb{L}^4+2\mathbb{L}^3+1\\
		e_\mathrm{c}(\mathcal{X}^{\times 2}_3,\mathbb{Q}_\ell) &= \mathbb{L}^{12}+4\mathbb{L}^{11}+10\mathbb{L}^{10}+20\mathbb{L}^9+25\mathbb{L}^8+24\mathbb{L}^7+17\mathbb{L}^6+\mathbb{L}^5-8\mathbb{L}^4-4\mathbb{L}^3-\mathbb{L}^2+4\mathbb{L}+5\\
		e_\mathrm{c}(\mathcal{X}^{\times 3}_3,\mathbb{Q}_\ell) &= \mathbb{L}^{15}+7\mathbb{L}^{14}+28\mathbb{L}^{13}+84\mathbb{L}^{12}+164\mathbb{L}^{11}+237\mathbb{L}^{10}+260\mathbb{L}^9\\
			&+ 164\mathbb{L}^8-21\mathbb{L}^7-171\mathbb{L}^6-212\mathbb{L}^5-107\mathbb{L}^4+47\mathbb{L}^3+99\mathbb{L}^2+75\mathbb{L}+29\\
		e_\mathrm{c}(\mathcal{X}^{\times 4}_3,\mathbb{Q}_\ell) &= \mathbb{L}^{18}+11\mathbb{L}^{17}+66\mathbb{L}^{16}+286\mathbb{L}^{15}+835\mathbb{L}^{14}+1775\mathbb{L}^{13}+2906\mathbb{L}^{12}+3480\mathbb{L}^{11}+2476\mathbb{L}^{10}\\
			&- 415\mathbb{L}^9-3846\mathbb{L}^8-5322\mathbb{L}^7-3781\mathbb{L}^6-597\mathbb{L}^5+2146\mathbb{L}^4+2877\mathbb{L}^3+1887\mathbb{L}^2+757\mathbb{L}+162\\
		e_\mathrm{c}(\mathcal{X}^{\times 5}_3,\mathbb{Q}_\ell) &= \mathbb{L}^{21}+16\mathbb{L}^{20}+136\mathbb{L}^{19}+816\mathbb{L}^{18}+3380\mathbb{L}^{17}+10182\mathbb{L}^{16}+23578\mathbb{L}^{15}\\
			&+ 42433\mathbb{L}^{14}+57157\mathbb{L}^{13}+47250\mathbb{L}^{12}-5213\mathbb{L}^{11}-84003\mathbb{L}^{10}-137082\mathbb{L}^9-124223\mathbb{L}^8\\
			&- 52325\mathbb{L}^7+33070\mathbb{L}^6+83756\mathbb{L}^5+83816\mathbb{L}^4+53066\mathbb{L}^3+22340\mathbb{L}^2+6134\mathbb{L}+891
	\end{align*}
	The compactly supported Euler characteristic \(e_\mathrm{c}(\mathcal{X}^{\times 6}_3,\mathbb{Q}_\ell)\) is given by:
	\begin{align*}
		e_\mathrm{c}(\mathcal{X}^{\times 6}_3,\mathbb{Q}_\ell) &= (\mathbb{L}^6+21\mathbb{L}^5+120\mathbb{L}^4+280\mathbb{L}^3+309\mathbb{L}^2+161\mathbb{L}+32)\mathbb{S}_{\Gamma(1)}[0,10]\\
			&+ \mathbb{L}^{24}+22\mathbb{L}^{23}+253\mathbb{L}^{22}+2024\mathbb{L}^{21}+11362\mathbb{L}^{20}+46613\mathbb{L}^{19}\\
			&+ 146665\mathbb{L}^{18}+364262\mathbb{L}^{17}+720246\mathbb{L}^{16}+1084698\mathbb{L}^{15}+1036149\mathbb{L}^{14}+38201\mathbb{L}^{13}\\
			&- 1876517\mathbb{L}^{12}-3672164\mathbb{L}^{11}-4024657\mathbb{L}^{10}-2554079\mathbb{L}^9+101830\mathbb{L}^8+2028655\mathbb{L}^7\\
			&+ 2921857\mathbb{L}^6+2536864\mathbb{L}^5+1553198\mathbb{L}^4+687157\mathbb{L}^3+215631\mathbb{L}^2+45035\mathbb{L}+4930
	\end{align*}
	where \(\mathbb{S}_{\Gamma(1)}[0,10]=\mathbb{S}_{\Gamma(1)}[18]+\mathbb{L}^9+\mathbb{L}^8\) is the \(4\)-dimensional \(\ell\)-adic Galois representation attached to the Saito-Kurokawa lift \(\chi_{10}\in S_{0,10}(\Gamma(1))\) of the weight \(18\) cusp form \(f_{18}=\Delta E_6\in S_{18}(\Gamma(1))\). In particular the  compactly supported Euler characteristic \(e_\mathrm{c}(\mathcal{X}^{\times n}_3,\mathbb{Q}_\ell)\) is not Tate type if \(n\geq 6\).
	\end{theorem}
	\begin{proof}
	Follows by combining \ref{Leray} and \ref{Kunneth} with \ref{EichlerShimura3}. In this case we computed the multiplicities \(m^{j,n}_\lambda\) with a SAGE program (available on request).
	
	To argue that \(e_\mathrm{c}(\mathcal{X}^{\times n}_3,\mathbb{Q}_\ell)\) is not Tate type if \(n\geq 6\) note that \(H^{24}(\mathcal{X}^{\times 9}_3,\mathbb{Q}_\ell)\) (which is not Tate type, owing to the \(8\)-dimensional contribution \(\mathbb{S}_{\Gamma(1)}[3,3,7]\) to \(H^6(\mathcal{A}_3,\mathbb{V}_{9,6,3})\), which decomposes into a \(1\)-dimensional contribution and an irreducible \(7\)-dimensional contribution) appears as a summand in \(H^{24}(\mathcal{X}^{\times n}_3,\mathbb{Q}_\ell)\) for all \(n\geq 9\) by the K\"unneth formula. This contribution cannot be cancelled in the Euler characteristic: since the contribution occurs in \(H^i(\mathcal{X}^{\times n}_3,\mathbb{Q}_\ell)\) for \(i\) even, any contribution leading to cancellation would have to occur in \(H^i(\mathcal{X}^{\times n}_3,\mathbb{Q}_\ell)\) for \(i\) odd. Since \(H^*(\mathcal{A}_3,\mathbb{V}_{\lambda_1,\lambda_2,\lambda_3})=0\) for \(\lambda_1+\lambda_2+\lambda_3>0\) odd, any contribution to \(H^i(\mathcal{X}^{\times n}_3,\mathbb{Q}_\ell)\) for \(i\) odd would have to come from a contribution to \(H^j(\mathcal{A}_3,\mathbb{V}_{\lambda_1,\lambda_2,\lambda_3})\) for \(j=1,3,5,7,9,11\), but there are no irreducible \(7\)-dimensional contributions in this case: the only irreducible \(7\)-dimensional contributions come from the contributions to \(H^6(\mathcal{A}_3,\mathbb{V}_{\lambda_1,\lambda_2,\lambda_3})\) predicted by \cite{GrossSavin}. The remaining cases \(n=7,8\) are checked by running the above computations further to see that the contribution \(\mathbb{S}_{\Gamma(1)}[0,10]\) persists. 
	\end{proof}
	
	Alternatively, note that \(H^{26}(\mathcal{X}^{\times 10}_3,\mathbb{Q}_\ell)\) (which is not Tate type, owing to the irreducible \(8\)-dimensional contributions \(\mathbb{S}_{\Gamma(1)}[2,2,6]\) and \(\mathbb{S}_{\Gamma(1)}[4,2,8]\) to \(H^6(\mathcal{A}_3,\mathbb{V}_{10,8,2})\) and \(H^6(\mathcal{A}_3,\mathbb{V}_{10,6,4})\) respectively, see \cite[Table 1, Table 2]{BFvdG3}) appears as a summand in \(H^{26}(\mathcal{X}^{\times n}_3,\mathbb{Q}_\ell)\) for all \(n\geq 10\) by the K\"unneth formula. This contribution cannot be cancelled in the Euler characteristic by the same argument as above: the only irreducible \(8\)-dimensional contributions come from the contribution \(\mathbb{S}^\mathrm{gen}[\lambda_1-\lambda_2,\lambda_2-\lambda_3,\lambda_3+4]\) to \(H^6(\mathcal{A}_3,\mathbb{V}_{\lambda_1,\lambda_2,\lambda_3})\). The remaining cases \(n=7,8,9\) are checked by running the above computations further to see that the contribution \(\mathbb{S}_{\Gamma(1)}[0,10]\) persists. This makes the above argument a bit less conjectural by removing the dependence on the functorial lift from \(\mathrm{G}_2\). That being said, since the above computations are already conditional on conjectures \ref{EichlerShimura3} and \ref{MiddleGalois}, we do not try to further justify the predictions of the Langlands correspondence which we have used in the above argument. 
	
	The contribution \((\mathbb{L}^6+21\mathbb{L}^5+120\mathbb{L}^4+280\mathbb{L}^3+309\mathbb{L}^2+161\mathbb{L}+32)\mathbb{S}_{\Gamma(1)}[0,10]\) to \(e_\mathrm{c}(\mathcal{X}^{\times 6}_3,\mathbb{Q}_\ell)\) comes from the following \(4\) contributions:
	\begin{align*}
		e_\mathrm{c}(\mathcal{A}_3,\mathbb{V}_{6,6,6}) &+ (15\mathbb{L}^2+35\mathbb{L}+15)e_\mathrm{c}(\mathcal{A}_3,\mathbb{V}_{6,6,4})\\
		&+ (15\mathbb{L}^4+105\mathbb{L}^3+189\mathbb{L}^2+105\mathbb{L}+15)e_\mathrm{c}(\mathcal{A}_3,\mathbb{V}_{6,6,2})\\
		&+ (\mathbb{L}^6+21\mathbb{L}^5+105\mathbb{L}^4+175\mathbb{L}^3+105\mathbb{L}^2+21\mathbb{L}+1)e_\mathrm{c}(\mathcal{A}_3,\mathbb{V}_{6,6,0})
	\end{align*}
	which explains why the coefficients in the polynomial \(\mathbb{L}^6+21\mathbb{L}^5+120\mathbb{L}^4+280\mathbb{L}^3+309\mathbb{L}^2+161\mathbb{L}+32\) are not symmetric: it arises as the sum of \(4\) polynomials with symmetric coefficients of different degrees. Note that the contribution \(\mathbb{S}_{\Gamma(1)}[0,10]\) should always persist, but we cannot argue this without estimates on the multiplicities \(m^{j,n}_\lambda\). 
	
	We obtain the following corollary:
	
	\begin{corollary}
	\label{MGF3} The first \(5\) terms of the moment generating function \(M_{\#A_3(\mathbb{F}_q)}(t)\) are rational functions in \(q\): 
	\begin{align*}
		1 &+ \scriptstyle(\mathbf{q^3+q^2+q+1}+\tfrac{-q^2-q}{q^6+q^5+q^4+q^3+1})t\\
		&+ \scriptstyle(\mathbf{q^6+3q^5+6q^4+10q^3}+6q^2+2q-2+\tfrac{-8q^5-14q^4-12q^3-7q^2+2q+7}{q^6+q^5+q^4+q^3+1})\tfrac{t^2}{2!}\\
		&+ (\begin{smallmatrix} \mathbf{q^9+6q^8+21q^7+56q^6}+81q^5\\ +79q^4+43q^3-45q^2-119q-106 \end{smallmatrix}+\tfrac{-23q^5+39q^4+110q^3+144q^2+194q+135}{q^6+q^5+q^4+q^3+1})\tfrac{t^3}{3!}\\
		&+ (\begin{smallmatrix} \mathbf{q^{12}+10q^{11}+55q^{10}+220q^9}+550q^8+950q^7+1185q^6\\ +785q^5-499q^4-2106q^3-2576q^2-1091q+807 \end{smallmatrix}+\tfrac{1478q^5+2929q^4+4176q^3+4463q^2+1848q-645}{q^6+q^5+q^4+q^3+1})\tfrac{t^4}{4!}\\
		&+ (\begin{smallmatrix} \mathbf{q^{15}+15q^{14}+120q^{13}+680q^{12}}+2565q^{11}+6817q^{10}+13515q^9+19521q^8\\ +17184q^7-3650q^6-40833q^5-63521q^4-42593q^3+3203q^2+33402q+42708 \end{smallmatrix}+\tfrac{45276q^5+71227q^4+52951q^3+19137q^2-27268q-41817}{q^6+q^5+q^4+q^3+1})\tfrac{t^5}{5!}
	\end{align*}
	\end{corollary}
	
	Note that the first \(4\) coefficients in each of these terms (in bold) are consistent with \ref{Conjecture1}. 
	
	\pagebreak

	\begin{landscape}
	\begin{table}
	\begin{center}\scalebox{0.55}{\begin{tabular}{|c|c|c|c|c|c|c|c|c|c|c|}
		\hline
		\(i\) & \(H^i(\mathcal{X}_1,\mathbb{Q}_\ell)\) & \(H^i(\mathcal{X}^{\times 2}_1,\mathbb{Q}_\ell)\) & \(H^i(\mathcal{X}^{\times 3}_1,\mathbb{Q}_\ell)\) & \(H^i(\mathcal{X}^{\times 4}_1,\mathbb{Q}_\ell)\) & \(H^i(\mathcal{X}^{\times 5}_1,\mathbb{Q}_\ell)\) & \(H^i(\mathcal{X}^{\times 6}_1,\mathbb{Q}_\ell)\) & \(H^i(\mathcal{X}^{\times 7}_1,\mathbb{Q}_\ell)\) & \(H^i(\mathcal{X}^{\times 8}_1,\mathbb{Q}_\ell)\) & \(H^i(\mathcal{X}^{\times 9}_1,\mathbb{Q}_\ell)\) & \(H^i(\mathcal{X}^{\times 10}_1,\mathbb{Q}_\ell)\)\\
		\hline
		\(0\) & \(1\) & \(1\) & \(1\) & \(1\) & \(1\) & \(1\) & \(1\) & \(1\) & \(1\) & \(1\)\\
		\(1\) & \(0\) & \(0\) & \(0\) & \(0\) & \(0\) & \(0\) & \(0\) & \(0\) & \(0\) & \(0\)\\
		\(2\) & \(\mathbb{L}\) & \(3\mathbb{L}\) & \(6\mathbb{L}\) & \(10\mathbb{L}\) & \(15\mathbb{L}\) & \(21\mathbb{L}\) & \(28\mathbb{L}\) & \(36\mathbb{L}\) & \(45\mathbb{L}\) & \(55\mathbb{L}\)\\
		\(3\) & & \(\mathbb{L}^3\) & \(3\mathbb{L}^3\) & \(6\mathbb{L}^3\) & \(10\mathbb{L}^3\) & \(15\mathbb{L}^3\) & \(21\mathbb{L}^3\) & \(28\mathbb{L}^3\) & \(36\mathbb{L}^3\) & \(45\mathbb{L}^3\)\\
		\(4\) & & \(\mathbb{L}^2\) & \(6\mathbb{L}^2\) & \(20\mathbb{L}^2\) & \(50\mathbb{L}^2\) & \(105\mathbb{L}^2\) & \(196\mathbb{L}^2\) & \(336\mathbb{L}^2\) & \(540\mathbb{L}^2\) & \(825\mathbb{L}^2\)\\
		\(5\) & & & \(3\mathbb{L}^4\) & \(\mathbb{L}^5+15\mathbb{L}^4\) & \(5\mathbb{L}^5+45\mathbb{L}^4\) & \(15\mathbb{L}^5+105\mathbb{L}^4\) & \(35\mathbb{L}^5+210\mathbb{L}^4\) & \(70\mathbb{L}^5+378\mathbb{L}^4\) & \(126\mathbb{L}^5+630\mathbb{L}^4\) & \(210\mathbb{L}^5+990\mathbb{L}^4\)\\
		\(6\) & & & \(\mathbb{L}^3\) & \(10\mathbb{L}^3\) & \(50\mathbb{L}^3\) & \(175\mathbb{L}^3\) & \(490\mathbb{L}^3\) & \(1176\mathbb{L}^3\) & \(2520\mathbb{L}^3\) & \(4950\mathbb{L}^3\)\\
		\(7\) & & & & \(6\mathbb{L}^5\) & \(5\mathbb{L}^6+45\mathbb{L}^5\) & \(\mathbb{L}^7+35\mathbb{L}^6+189\mathbb{L}^5\) & \(7\mathbb{L}^7+140\mathbb{L}^6+588\mathbb{L}^5\) & \(28\mathbb{L}^7+420\mathbb{L}^6+1512\mathbb{L}^5\) & \(84\mathbb{L}^7+1050\mathbb{L}^6+3420\mathbb{L}^5\) & \(210\mathbb{L}^7+2310\mathbb{L}^6+6930\mathbb{L}^5\)\\
		\(8\) & & & & \(\mathbb{L}^4\) & \(15\mathbb{L}^4\) & \(105\mathbb{L}^4\) & \(490\mathbb{L}^4\) & \(1764\mathbb{L}^4\) & \(5292\mathbb{L}^4\) & \(13860\mathbb{L}^4\)\\
		\(9\) & & & & & \(10\mathbb{L}^6\) & \(15\mathbb{L}^7+105\mathbb{L}^6\) & \(7\mathbb{L}^8+140\mathbb{L}^7+588\mathbb{L}^6\) & \(\mathbb{L}^9+63\mathbb{L}^8+720\mathbb{L}^7+2353\mathbb{L}^6\) & \(9\mathbb{L}^9+315\mathbb{L}^8+2700\mathbb{L}^7+7560\mathbb{L}^6\) & \(45\mathbb{L}^9+1155\mathbb{L}^8+8250\mathbb{L}^7+20790\mathbb{L}^6\)\\
		\(10\) & & & & & \(\mathbb{L}^5\) & \(21\mathbb{L}^5\) & \(196\mathbb{L}^5\) & \(1176\mathbb{L}^5\) & \(5292\mathbb{L}^5\) & \(19404\mathbb{L}^5\)\\
		\(11\) & & & & & & \(15\mathbb{L}^7\) & \(35\mathbb{L}^8+210\mathbb{L}^7\) & \(28\mathbb{L}^9+420\mathbb{L}^8+1512\mathbb{L}^7\) & \(9\mathbb{L}^{10}+315\mathbb{L}^9+2700\mathbb{L}^8+7560\mathbb{L}^7\) & \(\begin{matrix} \mathbb{S}_{\Gamma(1)}[12]+\mathbb{L}^{11}+99\mathbb{L}^{10}\\ +1925\mathbb{L}^9+12375\mathbb{L}^8+29700\mathbb{L}^7 \end{matrix}\)\\
		\(12\) & & & & & & \(\mathbb{L}^6\) & \(28\mathbb{L}^6\) & \(336\mathbb{L}^6\) & \(2520\mathbb{L}^6\) & \(13860\mathbb{L}^6\)\\
		\(13\) & & & & & & & \(21\mathbb{L}^8\) & \(70\mathbb{L}^9+378\mathbb{L}^8\) & \(84\mathbb{L}^{10}+1050\mathbb{L}^9+3420\mathbb{L}^8\) & \(45\mathbb{L}^{11}+1155\mathbb{L}^{10}+8250\mathbb{L}^9+20790\mathbb{L}^8\)\\
		\(14\) & & & & & & & \(\mathbb{L}^7\) & \(36\mathbb{L}^7\) & \(540\mathbb{L}^7\) & \(4950\mathbb{L}^7\)\\
		\(15\) & & & & & & & & \(28\mathbb{L}^9\) & \(126\mathbb{L}^{10}+630\mathbb{L}^9\) & \(210\mathbb{L}^{11}+2310\mathbb{L}^{10}+6930\mathbb{L}^9\)\\
		\(16\) & & & & & & & & \(\mathbb{L}^8\) & \(45\mathbb{L}^8\) & \(825\mathbb{L}^8\)\\
		\(17\) & & & & & & & & & \(36\mathbb{L}^{10}\) & \(210\mathbb{L}^{11}+990\mathbb{L}^{10}\)\\
		\(18\) & & & & & & & & & \(\mathbb{L}^9\) & \(55\mathbb{L}^9\)\\
		\(19\) & & & & & & & & & & \(45\mathbb{L}^{11}\)\\
		\(20\) & & & & & & & & & & \(\mathbb{L}^{10}\)\\
		\hline
	\end{tabular}}\end{center}
	\caption{\(H^i(\mathcal{X}^{\times n}_1,\mathbb{Q}_\ell)\) for \(1\leq n\leq 10\)}
	\end{table}
	\begin{table}
	\begin{center}\scalebox{0.56}{\begin{tabular}{|c|c|c|c|c|c|c|c|}
		\hline
		\(i\) & \(H^i(\mathcal{X}_2,\mathbb{Q}_\ell)\) & \(H^i(\mathcal{X}^{\times 2}_2,\mathbb{Q}_\ell)\) & \(H^i(\mathcal{X}^{\times 3}_2,\mathbb{Q}_\ell)\) & \(H^i(\mathcal{X}^{\times 4}_2,\mathbb{Q}_\ell)\) & \(H^i(\mathcal{X}^{\times 5}_2,\mathbb{Q}_\ell)\) & \(H^i(\mathcal{X}^{\times 6}_2,\mathbb{Q}_\ell)\) & \(H^i(\mathcal{X}^{\times 7}_2,\mathbb{Q}_\ell)\)\\
		\hline
		\(0\) & \(1\) & \(1\) & \(1\) & \(1\) & \(1\) & \(1\) & \(1\)\\
		\(1\) & \(0\) & \(0\) & \(0\) & \(0\) & \(0\) & \(0\) & \(0\)\\
		\(2\) & \(2\mathbb{L}\) & \(4\mathbb{L}\) & \(7\mathbb{L}\) & \(11\mathbb{L}\) & \(16\mathbb{L}\) & \(22\mathbb{L}\) & \(29\mathbb{L}\)\\
		\(3\) & \(0\) & \(0\) & \(0\) & \(0\) & \(0\) & \(0\) & \(0\)\\
		\(4\) & \(2\mathbb{L}^2\) & \(9\mathbb{L}^2\) & \(27\mathbb{L}^2\) & \(65\mathbb{L}^2\) & \(135\mathbb{L}^2\) & \(252\mathbb{L}^2\) & \(434\mathbb{L}^2\)\\
		\(5\) & \(\mathbb{L}^5\) & \(3\mathbb{L}^5+\mathbb{L}^4\) & \(6\mathbb{L}^5+3\mathbb{L}^4\) & \(10\mathbb{L}^5+6\mathbb{L}^4\) & \(15\mathbb{L}^5+10\mathbb{L}^4\) & \(21\mathbb{L}^5+15\mathbb{L}^4\) & \(28\mathbb{L}^5+21\mathbb{L}^4\)\\
		\(6\) & \(\mathbb{L}^3\) & \(9\mathbb{L}^3\) & \(49\mathbb{L}^3\) & \(191\mathbb{L}^3\) & \(590\mathbb{L}^3\) & \(1540\mathbb{L}^3\) & \(3542\mathbb{L}^3\)\\
		\(7\) & & \(4\mathbb{L}^6+3\mathbb{L}^5\) & \(21\mathbb{L}^6+18\mathbb{L}^5\) & \(66\mathbb{L}^6+60\mathbb{L}^5\) & \(160\mathbb{L}^6+150\mathbb{L}^5\) & \(330\mathbb{L}^6+315\mathbb{L}^5\) & \(609\mathbb{L}^6+588\mathbb{L}^5\)\\
		\(8\) & & \(4\mathbb{L}^4\) & \(49\mathbb{L}^4\) & \(326\mathbb{L}^4\) & \(1535\mathbb{L}^4\) & \(5698\mathbb{L}^4\) & \(17738\mathbb{L}^4\)\\
		\(9\) & & \(3\mathbb{L}^7+\mathbb{L}^6\) & \(\mathbb{L}^9+36\mathbb{L}^7+28\mathbb{L}^6\) & \(10\mathbb{L}^9+210\mathbb{L}^7+190\mathbb{L}^6\) & \(50\mathbb{L}^9+825\mathbb{L}^7+780\mathbb{L}^6\) & \(175\mathbb{L}^9+\mathbb{L}^8+2520\mathbb{L}^7+2415\mathbb{L}^6\) & \(490\mathbb{L}^9+7\mathbb{L}^8+6468\mathbb{L}^7+6216\mathbb{L}^6\)\\
		\(10\) & & \(\mathbb{L}^5\) & \(27\mathbb{L}^5\) & \(6\mathbb{L}^9+327\mathbb{L}^5\) & \(45\mathbb{L}^9+2457\mathbb{L}^5\) & \(189\mathbb{L}^9+13371\mathbb{L}^5\) & \(588\mathbb{L}^9+57540\mathbb{L}^5\)\\
		\(11\) & & & \(21\mathbb{L}^8+18\mathbb{L}^7\) & \(16\mathbb{L}^{10}+295\mathbb{L}^8+280\mathbb{L}^7\) & \(190\mathbb{L}^{10}+2085\mathbb{L}^8+2010\mathbb{L}^7\) & \(1155\mathbb{L}^{10}+21\mathbb{L}^9+9990\mathbb{L}^8+9570\mathbb{L}^7\) & \(4900\mathbb{L}^{10}+203\mathbb{L}^9+36995\mathbb{L}^8+35028\mathbb{L}^7\)\\
		\(12\) & & & \(7\mathbb{L}^6\) & \(15\mathbb{L}^{10}+191\mathbb{L}^6\) & \(225\mathbb{L}^{10}+2467\mathbb{L}^6\) & \(15\mathbb{L}^{11}+1539\mathbb{L}^{10}+20524\mathbb{L}^6\) & \(140\mathbb{L}^{11}+7014\mathbb{L}^{10}+125517\mathbb{L}^6\)\\
		\(13\) & & & \(6\mathbb{L}^9+3\mathbb{L}^8\) & \(10\mathbb{L}^{11}+210\mathbb{L}^9+190\mathbb{L}^8\) & \(\mathbb{L}^{13}+300\mathbb{L}^{11}+2850\mathbb{L}^9+2724\mathbb{L}^8\) & \(21\mathbb{L}^{13}+3255\mathbb{L}^{11}+105\mathbb{L}^{10}+22470\mathbb{L}^9+21294\mathbb{L}^8\) & \(196\mathbb{L}^{13}+21280\mathbb{L}^{11}+1568\mathbb{L}^{10}+124250\mathbb{L}^9+115542\mathbb{L}^8\)\\
		\(14\) & & & \(\mathbb{L}^7\) & \(6\mathbb{L}^{11}+65\mathbb{L}^7\) & \(351\mathbb{L}^{11}+1550\mathbb{L}^7\) & \(15\mathbb{L}^{13}+105\mathbb{L}^{12}+4536\mathbb{L}^{11}+20750\mathbb{L}^7\) & \(210\mathbb{L}^{13}+1470\mathbb{L}^{12}+32340\mathbb{L}^{11}+187063\mathbb{L}^7\)\\
		\(15\) & & & & \(66\mathbb{L}^{10}+60\mathbb{L}^9\) & \(190\mathbb{L}^{12}+2085\mathbb{L}^{10}+2010\mathbb{L}^9\) & \(36\mathbb{L}^{14}+4480\mathbb{L}^{12}+175\mathbb{L}^{11}+29295\mathbb{L}^{10}+27734\mathbb{L}^9\) & \(763\mathbb{L}^{14}+48580\mathbb{L}^{12}+4802\mathbb{L}^{11}+253526\mathbb{L}^{10}+233429\mathbb{L}^9\)\\
		\(16\) & & & & \(11\mathbb{L}^8\) & \(225\mathbb{L}^{12}+590\mathbb{L}^8\) & \(35\mathbb{L}^{14}+189\mathbb{L}^{13}+6426\mathbb{L}^{12}+13671\mathbb{L}^8\) & \(980\mathbb{L}^{14}+4872\mathbb{L}^{13}+76566\mathbb{L}^{12}+190394\mathbb{L}^8\)\\
		\(17\) & & & & \(10\mathbb{L}^{11}+6\mathbb{L}^{10}\) & \(50\mathbb{L}^{13}+825\mathbb{L}^{11}+780\mathbb{L}^{10}\) & \(21\mathbb{L}^{15}+3255\mathbb{L}^{13}+105\mathbb{L}^{12}+22470\mathbb{L}^{11}+21294\mathbb{L}^{10}\) & \(\begin{matrix} \mathbb{S}_{\Gamma(1)}[18]+\mathbb{L}^{17}+1176\mathbb{L}^{15}+63700\mathbb{L}^{13}\\ +6860\mathbb{L}^{12}+321048\mathbb{L}^{11}+294440\mathbb{L}^{10}+\mathbb{L}^9 \end{matrix}\)\\
		\(18\) & & & & \(\mathbb{L}^9\) & \(45\mathbb{L}^{13}+135\mathbb{L}^9\) & \(15\mathbb{L}^{15}+105\mathbb{L}^{14}+4536\mathbb{L}^{13}+5803\mathbb{L}^9\) & \(1520\mathbb{L}^{15}+7056\mathbb{L}^{14}+101136\mathbb{L}^{13}+131215\mathbb{L}^9\)\\
		\(19\) & & & & & \(160\mathbb{L}^{12}+150\mathbb{L}^{11}\) & \(1155\mathbb{L}^{14}+21\mathbb{L}^{13}+9990\mathbb{L}^{12}+9570\mathbb{L}^{11}\) & \(763\mathbb{L}^{16}+48580\mathbb{L}^{14}+4802\mathbb{L}^{13}+253526\mathbb{L}^{12}+233429\mathbb{L}^{11}\)\\
		\(20\) & & & & & \(16\mathbb{L}^{10}\) & \(15\mathbb{L}^{15}+1539\mathbb{L}^{14}+1540\mathbb{L}^{10}\) & \(980\mathbb{L}^{16}+4872\mathbb{L}^{15}+76566\mathbb{L}^{14}+60270\mathbb{L}^{10}\)\\
		\(21\) & & & & & \(15\mathbb{L}^{13}+10\mathbb{L}^{12}\) & \(175\mathbb{L}^{15}+\mathbb{L}^{14}+2520\mathbb{L}^{13}+2415\mathbb{L}^{12}\) & \(196\mathbb{L}^{17}+21280\mathbb{L}^{15}+1568\mathbb{L}^{14}+124250\mathbb{L}^{13}+115542\mathbb{L}^{12}\)\\
		\(22\) & & & & & \(\mathbb{L}^{11}\) & \(189\mathbb{L}^{15}+252\mathbb{L}^{11}\) & \(210\mathbb{L}^{17}+1470\mathbb{L}^{16}+32340\mathbb{L}^{15}+18228\mathbb{L}^{11}\)\\
		\(23\) & & & & & & \(330\mathbb{L}^{14}+315\mathbb{L}^{13}\) & \(4900\mathbb{L}^{16}+203\mathbb{L}^{15}+36995\mathbb{L}^{14}+35028\mathbb{L}^{13}\)\\
		\(24\) & & & & & & \(22\mathbb{L}^{12}\) & \(140\mathbb{L}^{17}+7014\mathbb{L}^{16}+3542\mathbb{L}^{12}\)\\
		\(25\) & & & & & & \(21\mathbb{L}^{15}+15\mathbb{L}^{14}\) & \(490\mathbb{L}^{17}+7\mathbb{L}^{16}+6468\mathbb{L}^{15}+6216\mathbb{L}^{14}\)\\
		\(26\) & & & & & & \(\mathbb{L}^{13}\) & \(588\mathbb{L}^{17}+434\mathbb{L}^{13}\)\\
		\(27\) & & & & & & & \(609\mathbb{L}^{16}+588\mathbb{L}^{15}\)\\
		\(28\) & & & & & & & \(29\mathbb{L}^{14}\)\\
		\(29\) & & & & & & & \(28\mathbb{L}^{17}+21\mathbb{L}^{16}\)\\
		\(30\) & & & & & & & \(\mathbb{L}^{15}\)\\
		\hline
	\end{tabular}}\end{center}
	\caption{\(H^i(\mathcal{X}^{\times n}_2,\mathbb{Q}_\ell)\) for \(1\leq n\leq 7\)}
	\end{table}
	\end{landscape}
	
\pagebreak	

\end{document}